\documentclass[12pt, reqno]{amsart}
\usepackage{amsmath,amssymb,amsthm}
\allowdisplaybreaks
\usepackage{amsmath, amssymb}
\usepackage{amsthm, amsfonts, mathrsfs}
\usepackage{mathptmx}
\usepackage{fullpage}
\usepackage{amsfonts,graphicx}
\usepackage{dutchcal}
\numberwithin{equation}{section}
\usepackage[colorlinks=true, pdfstartview=FitV, linkcolor=blue, citecolor=blue, urlcolor=blue]{hyperref}

\newtheorem{theo}{Theorem}[section]
\newtheorem{lemm}[theo]{Lemma}

\newtheorem{prop}[theo]{Proposition}
\newtheorem{rema}[theo]{Remark}
\numberwithin{equation}{section}

\newcommand{\R}{{\mathbb R}}

\begin{document}
\title{Global regularity and sharp decay to the 2D Hypo-Viscous compressible Navier-Stokes equations}

\author{Chen Liang}

\address{ School of Science,
	Shenzhen Campus of Sun Yat-sen University, Shenzhen 518107, China}

\email{liangch89@mail2.sysu.edu.cn}

\author{Zhaonan Luo}
\address{School of Science,
 Shenzhen Campus of Sun Yat-sen University, Shenzhen 518107, China}

\email{luozhn7@mail.sysu.edu.cn}

\author{Zhaoyang Yin}

\address{ School of Science,
Shenzhen Campus of Sun Yat-sen University, Shenzhen 518107, China}

\email{mcsyzy@mail.sysu.edu.cn }

%\footnote{\today}



\begin{abstract}
In this paper, we consider the global regularity and the optimal time decay rate for the 2D isentropic hypo-viscous compressible Navier-Stokes equations. Firstly, we prove that there exists a global strong solution with the small initial data are close to the constant equilibrium state in $H^s$ framework with $s>1$. Furthermore, by virtue of improved Fourier splitting method and the Littlewood-Paley decomposition theory, we then establish the optimal time decay rate for low regularity data.

\end{abstract}

\maketitle

\tableofcontents
\section{Introduction and the main results}
\subsection {Model and synopsis of related studies}

In this paper, we consider the following compressible Navier-Stokes (CNS) equations on the $\mathbb{R}^2$:
\begin{align}\label{eq0}
\left\{\begin{aligned} 
&\partial_{t}\rho+\mathrm{div}(\rho u)=0,\\ 
&\partial_{t}(\rho u)+\mu(-\Delta)^{\beta}u-(\mu+\nu)\nabla \mathrm{div}u+\mathrm{div}(\rho u\otimes u)+\nabla P(\rho)=0,\\
&\rho|_{t=0}=\rho_0,~~u|_{t=0}=u_0,
\end{aligned}\right.
\end{align}
where $\rho$ and $u$ are density and velocity, the pressure $P=\rho^{\gamma}\ (\gamma\ge1)$ is given a smooth function, $\mu>0$ and $\nu$ are  the shear viscosity and the bulk viscosity respectively satisfying the physical restrictions
\begin{align*}
 \mu+\nu\ge0.
 \end{align*}
  Fractional Laplacian operator $(-\Delta)^{\beta}u$ is defined by the Fourier transform
	$$
	\widehat{(-\Delta)^{\beta}u}=|\xi|^{2\beta}\widehat{u},\quad\quad \widehat{u}(\xi,t)=\mathscr{F}u(\xi,t)=\int_{\mathbb{R}^d}e^{-ix\cdot\xi}u(x,t) dx.
	$$
The CNS with $\beta=1$ is the classical barotropic CNS equations, which has been systemically studied in many articles. The existence of global strong solutions with small initial data in 3D has been obtained by Matsumura and Nishida \cite{1979Matsumura,1980Matsumura}. Moreover, they obtained the time decay rate under additionally assumption that initial data belong to $L^1$. More global well-posedness results of strong solutions can be found in \cite{MR1779621,MR2675485,MR2679372,MR2847531}.  
 Z. Xin \cite{xin1988blowup} proved there is no global non-trivial smooth solution for full CNS system without heat conduction if initial density is compactly supported. For the weak solutions of CNS equations, P. L. Lions \cite{1996Lions,1998Lions} proved there exists a global weak solution with finite energy. Later, Feireisl et al. did a series of works, which has enriched the theory of weak solutions for CNS equations, and readers may refer to \cite{2004Feireisl} for details.

Significant progress has been made over the past decades regarding the time decay rates for the multi-dimensional isentropic CNS equations with $\beta=1$. We assume that positive density tending to some positive constant at infinity (say 1, to simplify the notation).  Ponce \cite{PONCE1985399} proved the $L^{p}$ time decay rate with  small initial data in $H^{s}\cap W^{s,1}(s>2+\frac{d}{2})$. Y. Guo and Y. Wang \cite{Guo01012012} obtained the optimal time decay rates for the CNS equations and the Boltzmann equation with smallness assumption on $H^{s}(s>3)$, and initial data belongs to negative index Sobolev space $\dot{H}^{-s}(0\le s<\frac{3}{2})$. The method is purely based on energy estimates. R. Danchin and J. Xu \cite{MR3609245} proved the  time decay rates for CNS equations in the critical $L^{p}$ setting, where viscosity coefficients depend on $\rho$. Later, J. Xu \cite{Xu2019} established sharp time decay rate for CNS equations under small low frequency assumption in some Besov spaces with negative index. which improve the results from \cite{MR3609245}. Z. Xin and J. Xu \cite{MR4188989} prove same results, in particular, the restriction that the low frequency of initial data is enough small was removed
there. Spectral analysis is a important method to analysis the large-time behavior. One may refer to \cite{Li2011Large,DUAN2007220,1979Matsumura} for more details.  

When $0<\beta<1$, Y. Li, P. Qu, Z. Zeng and D. Zhang \cite{li2022} proved there exists the non-uniqueness of weak solutions for CNS equations (\ref{eq0}). S. Wang and S. Zhang \cite{MR4643428} studied the $L^2$ time decay rate of (\ref{eq0}) with the case $\mu+\nu=0$ and $\frac{1}{2}<\beta\le1$ in $\mathbb{R}^3$, where the initial data belong to $H^{4}\cap L^{1}$. 
Recently, S. Wang and S. Zhang \cite{MR4897629} obtained the time decay rate for the 3D Navier-Stokes-Poisson system with fractional dissipation, which can be regarded as an extension from  \cite{MR4643428}.

\subsection{Main results}
To our best knowledge, large time behavior for the CNS equations (\ref{eq0}) has not been
studied yet. In this paper, we mainly study the optimal time decay rate of global strong solutions for CNS equations (\ref{eq0}) with the case with $\mu=1$ and $\nu=-1$. Taking $\rho=a+1$, (\ref{eq0}) can be written as
\begin{align}\label{eq1}
\left\{\begin{aligned} 
&\partial_{t}a+\mathrm{div}u=-\mathrm{div}(au),\\ 
&\partial_{t}u+(-\Delta)^{\beta}u+\gamma\nabla a=K(a)\nabla a+G,\\
&a|_{t=0}=a_0,~~u|_{t=0}=u_0,
\end{aligned}\right.
\end{align}
where $$
K(a)=\gamma\frac{a}{1+a}+\frac{P^{\prime}(1)- P^{\prime}(1+a)}{1+a}, \\
$$
and
$$
G=-u\cdot\nabla u+\frac{a}{1+a}(-\Delta)^{\beta}u.
$$

We firstly establish global regularity for the solutions of (\ref{eq1}) with small data in $H^{s}(s>1)$ by standard energy estimate method and bootstrap method. Precisely, our first result states as follows:
\begin{theo}\label{1theo1}
Let $\frac{1}{2}\le\beta<1$, $s>1$. Let $(a,u)$ be a local strong solution of (\ref{eq1}) with the 
initial data $(a_{0},u_{0})\in H^{s}$. There exists a small constant $\delta $ such that if 
\begin{align*}
\|(a_{0},u_{0})\|_{H^{s}}\leq\delta,
\end{align*}
then the system (\ref{eq1}) admits a unique global strong solution $(a,u)\in C([0,\infty),H^{s})$. Moreover, we obtain thar for all $t>0$, there holds
\begin{align*}
\quad\frac{d}{dt}\left(\|(\sqrt{\gamma}a,u)\|_{H^{s}}^{2}+2k\langle\Lambda^{\beta-1}\nabla a,\Lambda^{\beta-1}u\rangle_{H^{s-2\beta+1}}\right)+k\gamma\|\Lambda^{\beta} a\|_{H^{s-2\beta+1}}^{2}+\|\Lambda^{\beta}u\|_{H^{s}}^{2}\leq 0,
\end{align*}
where $k$ is a sufficiently small constant.
\end{theo}
\begin{rema}
It should be noticed that there exists parabolic effects for $a$ and $u$, This phenomenon is unique to the case with $\frac{1}{2}\le\beta<1$.
\end{rema}
We shall study the time decay rate of the global small solutions to the system (\ref{eq1}) based on the Schonbek's \cite{Schonbek1985,Schonbek1991} strategy. By virtue of the improved Fourier splitting method, we obtain initial decay rate 
\begin{align*}
\|(a,u)\|_{H^s}\lesssim(1+t)^{-\frac{1}{4\beta}+\frac{1}{4}}.
\end{align*}
By virtue of the time weighted energy estimate, we improve the time decay rate to
\begin{align*}
\|(a,u)\|_{H^s}\lesssim(1+t)^{-\frac{1}{4\beta}}.
\end{align*}
Then by Littlewood-Paley theory and time decay rate of $(a,u)$, we can prove $(a,u)\in L^{ \infty}([0,\infty),\dot{B}_{2,\infty}^{-1})$. Finally, we obtain the optimal time decay rate
\begin{align*}
\|(a,u)\|_{H^s}\lesssim(1+t)^{-\frac{1}{2\beta}},
\end{align*}
which gives a rise to
\begin{align}\label{101}
(1+t)^{-1}\int_{0}^{t}(1+t^{\prime})^{\frac{1}{\beta}+1}\left(\|\Lambda^{\beta}a\|_{H^{s+1-2\beta}}^{2}+\|\Lambda^{\beta} u\|_{H^{s}}^{2}\right)dt^{\prime}\le C.
\end{align}
 We also want to use same method get $\beta$-order time decay rate. However, we can not closed the energy in $H^{s-\beta}$ because of  inner product estimate. Fortunately, we find (\ref{421}) holds,  which combines (\ref{101}) derive the $\beta$-order time decay rate. For highest order time decay rate, we construct time weighted energy functional and dissipation functional, which will help us close the energy estimate. By virtue of Fourier splitting method and time weighted energy estimate, one can derive the upper bound of decay rate for the highest derivative. Finally, we introduce a new weighted energy estimate instead of complex
spectral analysis to prove the lower bound of the decay rate. Our second results can be written as follows:
\begin{theo}\label{1theo2}
Let $\frac{1}{2}\le\beta<1$, $s>1$. Let $(a,u)$ be a strong solution of (\ref{eq1}) with the initial data  $(a_{0}, u_{0})$ under the condition in Theorem \ref{1theo1}. If additionally $\left(a_{0}, u_{0}\right) \in \dot{B}_{2, \infty}^{-1}$, then there exists  $C>0$  such that for every $t>0$ and $s_1\in[0,s]$,
\begin{align*}
\left\|\Lambda^{s_1}(a,u)\right\|_{L^{2}} \leq C(1+t)^{-\frac{s_1+1}{2 \beta}}.
\end{align*}
Furthermore, suppose that  $\left(a_{0}, u_{0}\right) \in H^{s+\beta}$  and  $0<\left|\int_{\mathbb{R}^{2}}\left(a_{0}, u_{0}\right) d x\right|$ , then there exists  $C_{\beta} \leq C $ such that
\begin{align*}
\left\|\Lambda^{s_1}(a,u)\right\|_{L^{2}} \geq \frac{C_{\beta}}{2}(1+t)^{-\frac{s_1+1}{2 \beta}}.
\end{align*}
\end{theo}
\begin{rema}
  Noting $L^1\hookrightarrow\dot{B}_{2, \infty}^{-1}$, it follows that the above results still hold true when $\left(a_{0}, u_{0}\right)\in L^1$.
\end{rema}

\textbf{Organization of the paper}
The paper is organized as follows. In Section 2 we introduce some lemmas which will be used
in the sequel. In Section 3 we prove the global regularity for (\ref{eq1}).
In Section 4 we will prove optimal time decay rate for the (\ref{eq1}) by the improved Fourier splitting method and the time weighted
energy estimate.

\vskip .5in
\section{Preliminaries}

We agree that $f\lesssim g$ represents $f\leq Cg$ with a constant $C>0$. The symbol $\widehat{f}=\mathscr{F}(f)$ stands for the Fourier transform of $f$. Denote $\mathscr{F}^{-1}(f)$ the inverse Fourier transform of $f$.
 Let $A$, $B$ be two operators, we denote $[A, B] = AB - BA$, the commutator
between $A$ and $B$. Denote $\langle f,g\rangle$ the $L^2(\R^d)$ inner product of $f$ and $g$.  For $X$ a Banach space and $I$ an interval of $\mathbb{R}$, for any $f,g,h\in X$, we agree that
$\left\|\left(f,g\right)\right\| _{X}\stackrel{\mathrm{def}}{=} \left\|f\right\| _{X}
+\left\|g\right\|_{X}$
and denote by $C(I; X)$ the set of
continuous functions on $I$ with values in $X$.

	The Littlewood-Paley decomposition theory and Besov spaces are given as follows.
	\begin{lemm}\cite{Bahouri2011}\label{lemma1}
		Let $\mathscr{C}$ be the annulus $\{\xi\in\mathbb{R}^d:\frac 3 4\leq|\xi|\leq\frac 8 3\}$. There exists a radial functions $\varphi$, valued in the interval $[0,1]$, belonging to $\mathscr{D}(\mathscr{C})$, and such that
		%$$ \forall\xi\in\mathbb{R}^d,\ \chi(\xi)+\sum_{j\geq 0}\varphi(2^{-j}\xi)=1, $$
		$$ \forall\xi\in\mathbb{R}^2\backslash\{0\},\ \sum_{j\in\mathbb{Z}}\varphi(2^{-j}\xi)=1,~~~ $$
		$$ |j-j'|\geq 2\Rightarrow\mathrm{Supp}\ \varphi(2^{-j}\cdot)\cap \mathrm{Supp}\ \varphi(2^{-j'}\cdot)=\emptyset, $$
		%$$ ~~j\geq 1\Rightarrow\mathrm{Supp}\ \chi(\cdot)\cap \mathrm{Supp}\ \varphi(2^{-j}\cdot)=\emptyset. $$
		%The set $\widetilde{\mathscr{C}}=B(0,\frac 2 3)+\mathscr{C}$ is an annulus, then
		%$$ |j-j'|\geq 5\Rightarrow 2^{j}\mathscr{C}\cap 2^{j'}\widetilde{\mathscr{C}}=\emptyset. $$
		Moreover, we have
		%$$ ~~\forall\xi\in\mathbb{R}^d,\ \frac 1 2\leq\chi^2(\xi)+\sum_{j\geq 0}\varphi^2(2^{-j}\xi)\leq 1, $$
		$$ \forall\xi\in\mathbb{R}^2\backslash\{0\},\ \frac 1 2\leq\sum_{j\in\mathbb{Z}}\varphi^2(2^{-j}\xi)\leq 1.~~ $$
	\end{lemm}
			
		Let $u$ be a tempered distribution in $\mathcal{S}_{h}'(\mathbb{R}^2)$. For all $j\in\mathbb{Z}$, define
		%$$
		%\Delta_j u=0\,\ \text{if}\,\ j\leq -2,\quad
		%\Delta_{-1} u=\mathscr{F}^{-1}(\chi\mathscr{F}u),$$
		%$$\Delta_j u=\mathscr{F}^{-1}(\varphi(2^{-j}\cdot)\mathscr{F}u)\,\ \text{if}\,\ j\geq 0,\quad
		%S_j u=\sum_{j'<j}\Delta_{j'}u.
		%$$
		The homogeneous operators are defined by
		$$\dot{\Delta}_j u=\mathscr{F}^{-1}(\varphi(2^{-j}\cdot)\mathscr{F}u).$$
Then the Littlewood-Paley decomposition is given as follows:
$$
u=\sum_{j\in\mathbb{Z}}\dot{\Delta}_j u\quad  \mathrm{in}\quad  \mathcal{S}_{h}'(\mathbb{R}^2). 
$$		
		Let $s\in\mathbb{R}$ and $(p,r)\in[1,\infty]^2$. The homogeneous Besov space $\dot{B}^s_{p,r}$ is given as follows
	$$ \dot{B}^s_{p,r}=\left\{u\in \mathcal{S}_{h}'(\mathbb{R}^2):\|u\|_{\dot{B}^s_{p,r}}=\Big\|(2^{js}\|\dot{\Delta}_j u\|_{L^p})_j \Big\|_{l^r(\mathbb{Z})}<\infty\right\}.$$
 We introduce the Gagliardo-Nirenberg inequality of Sobolev type with
$d = 2$.
\begin{lemm}\label{lemma3}\cite{1959On}
For $d=2$, $p\in[2,+\infty)$ and $0\leq s,s_{1}\leq s_{2}$, there holds
\begin{align*}
\|\Lambda^{s}f\|_{L^{p}}\leq C\|\Lambda^{s_{1}}f\|_{L^{2}}^{1- \theta}\|\Lambda^{s_{2}}f\|_{L^{2}}^{\theta},
\end{align*}
where $0\leq\theta\leq 1$ and satisfies  
$$s+1-\frac{2}{p}=(1-\theta)s_{1}+\theta s_{2}.$$ 
Note that we also require that $0<\theta<1, 0\leq s_{1}\leq s$, when $ p=\infty$.
\end{lemm}

The following commutator estimate and product estimate are useful to estimate convective terms.
\begin{lemm}\label{lemma4}\cite{kato1}
Assume that $s>0$, $p, p_1, p_4 \in (1,\infty)$ and $\frac{1}{p}=\frac{1}{p_1}+\frac{1}{p_2}=\frac{1}{p_3}+\frac{1}{p_4}$, then we obtain
\begin{align*}
\|[\Lambda^{s},f]g\|_{L^{p}}\leq C\left(\|\Lambda^{s}f\|_{L^{p_{1}}}\|g\|_{L^{p_{2} }}+\|\nabla f\|_{L^{p_{3}}}\|\Lambda^{s-1}g\|_{L^{p_{4}}}\right).
\end{align*}
\end{lemm}
\begin{lemm}\label{lemma5}\cite{kato1}
Assume that $s>0$, $p, p_2, p_4 \in (1,\infty)$ and $\frac{1}{p}=\frac{1}{p_1}+\frac{1}{p_2}=\frac{1}{p_3}+\frac{1}{p_4}$, then we obtain
\begin{align*}
\|\Lambda^{s}\left(fg\right)\|_{L^{p}} \leq C\left(\|f \|_{L^{p_{1}}}\|\Lambda^{s}g\|_{L^{p_{2}}} + \|g\|_{L^{p_ {3}}}\|\Lambda^{s}f\|_{L^{p_{4}}}\right).
\end{align*}
\end{lemm}

\vskip .1in

\section{The global regularity to the global small solution to (\ref{eq1})}
In this section, we shall prove the critical global regularity for (\ref{eq1}). We set the following energy and dissipation functionals for $(a,u)$:
\begin{align*}
E_{0}&=\left\|(\sqrt{\gamma} a, u)\right\|_{H^{s}}^{2}+2k\langle\Lambda^{\beta-1}\nabla a,\Lambda^{\beta-1}u\rangle_{H^{s-2\beta+1}},\\
D_{0}&=k\gamma\|\Lambda^{\beta}a\|_{H^{s+1-2\beta}}^{2}+\|\Lambda^{\beta} u\|_{H^{s}}^{2},
\end{align*}
where $k$ is a sufficiently small constant.

\textbf{Proof of Theorem \ref{1theo1}:}

We assume that $(a,u)$ be a local strong solution of (\ref{eq1}), Before proceeding any further, we assume a
priori that
\begin{align*}
\|(a,u)\|_{H^{{s}}}\le \delta.
\end{align*}
Taking the $L^2$ inner product of the first two equations of (\ref{eq1}) with $(\gamma a,u)$ and integrating by parts, we get
\begin{align}\label{301}
\frac{\gamma }{2}\frac{d}{dt}\|a\|_{L^2}^{2}+\gamma\langle\mathrm{div}u,a\rangle=&-\gamma\langle\mathrm{div}(au),a\rangle\\\nonumber
\lesssim &\|\nabla a\|_{L^2}\|a\|_{L^4}\|u\|_{L^4}\\\nonumber
\lesssim &\|\nabla a\|_{L^2}\|a\|_{L^2}^{\frac{1}{2}}\|u\|_{L^2}^{\frac{1}{2}}\|\nabla u\|_{L^2}^{\frac{1}{2}}\|\nabla a\|_{L^2}^{\frac{1}{2}}\\\nonumber
\lesssim &\|\nabla a\|_{L^2}^{2}\|a\|_{L^2}+\|u\|_{L^2}\|\nabla u\|_{L^2}\|\nabla a\|_{L^2},
 \end{align}
 and
\begin{align}\label{302}
&\frac{1}{2}\frac{d}{dt}\|u\|_{L^2}^{2}+\gamma\langle\nabla a,u\rangle+\|\Lambda^{\beta}u\|_{L^2}^2\\\nonumber
=&\langle k(a)\nabla a,u\rangle-\langle u\cdot\nabla
 u,u\rangle+\left\langle\frac{a}{1+a}(-\Delta)^{\beta}u,u\right\rangle\\\nonumber
  \lesssim& \|\nabla a\|_{L^2}^{2}\|a\|_{L^2}+\|u\|_{L^2}\|\nabla u\|_{L^2}\|\nabla a\|_{L^2}+\|\nabla u\|_{L^2}^{2}\|u\|_{L^2}\\\nonumber
  &+\|\Lambda^{2\beta} u\|_{L^2}\|\nabla u\|_{L^2}\|u\|_{L^2}+\|\Lambda^{2\beta} u\|_{L^2}\|\nabla a\|_{L^2}\|a\|_{L^2}.
\end{align}
Combining (\ref{301}) and (\ref{302}), we have
\begin{align}\label{303}
\frac{1}{2}\frac{d}{dt}\|(\sqrt{\gamma} a,u)\|_{L^2}^{2}+\|\Lambda^{\beta}u\|_{L^2}^2+\|\mathrm{div}u\|_{L^2}^2\lesssim \delta D_0.
\end{align}
Applying $\Lambda^{s}$ to (\ref{eq1}), and taking the $L^2$ inner product of the first two equations of (\ref{eq1}) with $(\gamma\Lambda^{s}a,\Lambda^{s} u)$ and integrating by parts, by Lemma \ref{lemma4} and \ref{lemma5}, we obtain
\begin{align}\label{304}
&\frac{1}{2}\frac{d}{dt}\|\sqrt{\gamma} \Lambda^{s}a\|_{L^2}^{2}+\gamma\langle\Lambda^{s}\mathrm{div}u,\Lambda^{s}a\rangle\\\nonumber
=&-\gamma\langle\Lambda^{s}\mathrm{div}(au),\Lambda^{s}a\rangle\\\nonumber
=&-\gamma\langle[\Lambda^{s},u\cdot\nabla]a,\Lambda^{s}a\rangle-\gamma\langle u\cdot\nabla \Lambda^{s}a,\Lambda^{s}a\rangle-\gamma\langle\Lambda^{s}(a\mathrm{div}u),\Lambda^{s}a\rangle\\\nonumber
=&-\gamma\langle[\Lambda^{s},u\cdot\nabla]a,\Lambda^{s}a\rangle+\frac{\gamma}{2}\langle \mathrm{div}u, |\Lambda^{s}a|^{2}\rangle-\gamma\langle\Lambda^{s+\beta-1}(a\mathrm{div}u),\Lambda^{s-\beta+1}a\rangle\\\nonumber
\lesssim &\|\Lambda^{s} u\|_{L^{\frac{2}{1-\beta}}}\|\Lambda^{s} a\|_{L^{\frac{2}{\beta}}}\|\nabla a\|_{L^{2}}+\|\Lambda^{s} a\|_{L^{\frac{2}{\beta}}}\|\Lambda^{s} a\|_{L^{2}}\|\nabla u\|_{L^{\frac{2}{1-\beta}}}\\\nonumber
&+\|\Lambda^{s+\beta}u \|_{L^{2}}\|\Lambda^{s-\beta+1} a\|_{L^{2}}\|a\|_{L^{\infty}}+\|\mathrm{div}u \|_{L^{\frac{2}{1-\beta}}}\|\Lambda^{s-\beta+1} a\|_{L^{2}}\|\Lambda^{s+\beta-1}a\|_{L^{\frac{2}{\beta}}}\\\nonumber
\lesssim &\|\Lambda^{s+\beta} u\|_{L^{2}}\|\Lambda^{s-\beta+1} a\|_{L^{2}}\|\nabla a\|_{L^{2}}+\|\Lambda^{s-\beta+1} a\|_{L^{2}}\|\Lambda^{s} a\|_{L^{2}}\|\Lambda^{1+\beta}u\|_{L^{2}}\\\nonumber
&+\|\Lambda^{s+\beta}u \|_{L^{2}}\|\Lambda^{s-\beta+1} a\|_{L^{2}}\|a\|_{L^{\infty}}.
 \end{align}
 Along the same line, we have
\begin{align*}
&\frac{1}{2}\frac{d}{dt}\|\Lambda^{s}u\|_{L^2}^{2}+\gamma\langle\Lambda^{s}\nabla a,\Lambda^{s}u\rangle+\|\Lambda^{s+\beta}u\|_{L^2}^2\\
=&\langle \Lambda^{s} (k(a)\nabla a),\Lambda^{s}u\rangle-\langle \Lambda^{s}(u\cdot\nabla
 u),\Lambda^{s}u\rangle+\left\langle\Lambda^{s}\left(\frac{a}{1+a}(-\Delta)^{\beta}u\right),\Lambda^{s} u\right\rangle\\
 \triangleq& \sum_{i=1}^{3}I_{i}.
\end{align*}
Similar to the derivation of (\ref{304}), due to Lemma \ref{lemma5}, we have
\begin{align*}
I_1\lesssim& \|\Lambda^{s+\beta}u\|_{L^2}(\|\Lambda^{s-\beta+1}a\|_{L^2}\|k(a)\|_{L^\infty}+\|\Lambda^{s-\beta}k(a)\|_{L^{\frac{2}{1-\beta}}}\|\nabla a\|_{L^{\frac{2}{\beta}}})\\
\lesssim& \|\Lambda^{s+\beta}u\|_{L^2}(\|\Lambda^{s-\beta+1}a\|_{L^2}\|k(a)\|_{L^\infty}+\|\Lambda^{s}k(a)\|_{L^{2}}\|\Lambda^{2-\beta} a\|_{L^{2}})\\
\lesssim& \|\Lambda^{s+\beta}u\|_{L^2}(\|\Lambda^{s-\beta+1}a\|_{L^2}\|a\|_{L^\infty}+\|\Lambda^{s}a\|_{L^{2}}\|\Lambda^{2-\beta} a\|_{L^{2}}).
\end{align*} 
and
\begin{align*}
I_2\lesssim& \|\Lambda^{s+\beta}u\|_{L^2}\|\Lambda^{s-\beta+1}u\|_{L^2}\|\nabla u\|_{L^2}.
\end{align*}
Moreover, by Lemma \ref{lemma5}, we obtain
\begin{align*}
I_3\lesssim&\|\Lambda^{s+\beta}u\|_{L^2}(\|\Lambda^{s-\beta}a\|_{L^{\frac{2}{1-\beta}}}\|\Lambda^{2\beta}u\|_{L^{\frac{2}{\beta}}}+\|\Lambda^{s+\beta}u\|_{L^{2}}\|a \|_{L^{\infty}})\\
\lesssim&\|\Lambda^{s+\beta}u\|_{L^2}(\|\Lambda^{s}a\|_{L^{2}}\|\Lambda^{1+\beta}u\|_{L^{2}}+\|\Lambda^{s+\beta}u\|_{L^{2}}\|a \|_{L^{\infty}}).
\end{align*}
Combining the estimates from $I_1$ to $I_3$, we have
\begin{align}\label{305}
\frac{1}{2}\frac{d}{dt}\|\Lambda^{s}u\|_{L^2}^{2}+\gamma\langle\Lambda^{s}\nabla a,\Lambda^{s}u\rangle+\|\Lambda^{s+\beta}u\|_{L^2}^2\lesssim&\ \delta D_0.
\end{align}

Now we use the inner product between $a$ and $u$ to generate the dissipation of $a$. we choose $k>0$, which will be determined later. 
A simple calculation assures that
\begin{align}\label{306}
&\frac{d}{dt}\langle u,k\Lambda^{-2+2 \beta} \nabla a\rangle+k\gamma\|\Lambda^{\beta} a\|_{L^{2}}^{2} \\ \nonumber
=&k\langle\mathrm{div}\Lambda^{2\beta-2}u,\mathrm{div}u\rangle+k\langle\mathrm{div}\Lambda^{2\beta-2}u,\mathrm{div}au\rangle-k\langle(-\Delta)^{\beta}u,\Lambda^{2\beta-2}\nabla a\rangle\\\nonumber
&+k\langle k(a)\nabla a,\Lambda^{2\beta-2}\nabla a\rangle+k\left\langle-u\cdot\nabla u+\frac{a}{1+a}(-\Delta)^{\beta}u,\Lambda^{2\beta-2}\nabla a\right\rangle\\\nonumber
\le& k\|\Lambda^{\beta}u\|_{L^{2}}^{2}+Ck\|\Lambda^{2\beta-1}u\|_{L^{\frac{2}{\beta}}}(\|\nabla u\|_{L^2}\|a\|_{\frac{2}{1-\beta}}+\|\nabla a\|_{L^2}\|u\|_{\frac{2}{1-\beta}})+k\|\Lambda^{3\beta-1}u\|_{L^{2}}\|\Lambda^{\beta}a\|_{L^{2}}\\\nonumber
&+Ck\|\Lambda^{2\beta-1}a\|_{L^{\frac{2}{\beta}}}(\|\nabla a\|_{L^2}\|a\|_{\frac{2}{1-\beta}}+\|\nabla u\|_{L^2}\|u\|_{\frac{2}{1-\beta}})+Ck\|\Lambda^{2\beta-1}a\|_{L^{\frac{2}{\beta}}}\|\Lambda^{2\beta} u\|_{L^2}\|a\|_{\frac{2}{1-\beta}}\\\nonumber
\le & k\|\Lambda^{\beta}u\|_{L^{2}}^{2}+Ck\|\Lambda^{\beta}u\|_{L^{2}}(\|\nabla u\|_{L^2}\|\Lambda^{\beta}a\|_{2}+C\|\nabla a\|_{L^2}\|\Lambda^{\beta}u\|_{2})+k\|\Lambda^{3\beta-1}u\|_{L^{2}}\|\Lambda^{\beta}a\|_{L^{2}}\\\nonumber
&+Ck\|\Lambda^{\beta}a\|_{L^{2}}^{2}\|\nabla a\|_{L^2}+Ck\|\Lambda^{\beta}a\|_{L^{2}}\|\Lambda^{2\beta} u\|_{L^2}\|\Lambda^{\beta}a\|_{2}\\\nonumber
\le &C\delta D_0+\frac{k}{100}\|\Lambda^{\beta}a\|_{L^2}^{2}+Ck(\|\Lambda^{3\beta-1}u\|_{L^2}^{2}+\|\Lambda^{\beta}u\|_{L^2}^{2}).
\end{align}
Along the same line, we find
\begin{align}\label{307}
&\frac{d}{dt}\langle\Lambda^{s-\beta} \nabla a,k\Lambda^{s-\beta}u\rangle+k\gamma\|\Lambda^{s-\beta}\nabla a\|_{L^2}^2\\\nonumber
=& k\langle\Lambda^{s-\beta}\mathrm{div}u,\Lambda^{s-\beta}\mathrm{div}u\rangle+k\langle\Lambda^{s-\beta}\mathrm{div}au,\Lambda^{s-\beta}\mathrm{div}u\rangle-k\langle\Lambda^{s-\beta}(-\Delta)^{\beta}u,\Lambda^{s-\beta}\nabla a\rangle\\\nonumber
&+k\langle \Lambda^{s-\beta}(K(a)\nabla a),\Lambda^{s-\beta}\nabla a\rangle-k\langle \Lambda^{s-\beta}(u\cdot\nabla u),\Lambda^{s-\beta}\nabla a\rangle\\\nonumber
&+k\left\langle \Lambda^{s-\beta}\frac{a}{1+a}(-\Delta)^{\beta}u,\Lambda^{s-\beta}\nabla a\right\rangle\\\nonumber
=&\sum_{i=4}^{9}I_{i}\\\nonumber
\end{align}
Then we obtain
\begin{align*}
  I_4\lesssim k\|\Lambda^{s-\beta}\mathrm{div}u\|_{L^2}^2.
\end{align*}
By Lemma \ref{lemma5}, we have
\begin{align*}
I_5\lesssim k\|\Lambda^{s-\beta+1}u\|_{L^{2}}(\|u\|_{L^{\infty}}\|\Lambda^{s-\beta+1}a\|_{L^{2}}+\|a\|_{L^{\infty}}\|\Lambda^{s-\beta+1}u\|_{L^{2}}).
\end{align*}
Using H\"older's inequality, we get
\begin{align*}
  I_6\le k\|\Lambda^{s+\beta}u\|_{L^2}\|\Lambda^{s-\beta+1}a\|_{L^2}.
\end{align*}
Thanks to Lemma \ref{lemma5}, which gives rise to
\begin{align*}
  I_7 &\lesssim k\|\Lambda^{s-\beta+1}a\|_{L^{2}}(\|a\|_{L^{\infty}}\|\Lambda^{s-\beta+1}a\|_{L^{2}}+\|\nabla a\|_{L^{\frac{2}{\beta}}}\|\Lambda^{s-\beta}a\|_{L^{\frac{2}{1-\beta}}})\\
  &\lesssim k\|\Lambda^{s-\beta+1}a\|_{L^{2}}(\|a\|_{L^{\infty}}\|\Lambda^{s-\beta+1}a\|_{L^{2}}+\|\Lambda^{2-\beta} a\|_{L^{2}}\|\Lambda^{s}a\|_{L^{2}}),
\end{align*}
and
\begin{align*}
 I_8 &\lesssim k\|\Lambda^{s-\beta+1}a\|_{L^{2}}(\|u\|_{L^{\infty}}\|\Lambda^{s-\beta+1}u\|_{L^{2}}+\|\nabla u\|_{L^{\frac{2}{\beta}}}\|\Lambda^{s-\beta}u\|_{L^{\frac{2}{1-\beta}}})\\
 &\lesssim k\|\Lambda^{s-\beta+1}a\|_{L^{2}}(\|u\|_{L^{\infty}}\|\Lambda^{s-\beta+1}u\|_{L^{2}}+\|\Lambda^{2-\beta} u\|_{L^{2}}\|\Lambda^{s}u\|_{L^{2}}),
\end{align*}
and
\begin{align*}
I_9 &\lesssim k\|\Lambda^{s-\beta+1}a\|_{L^{2}}(\|a\|_{L^{\infty}}\|\Lambda^{s+\beta}u\|_{L^{2}}+\|\Lambda^{2\beta} u\|_{L^{\frac{2}{\beta}}}\|\Lambda^{s-\beta}a\|_{L^{\frac{2}{1-\beta}}})\\
&\lesssim k\|\Lambda^{s-\beta+1}a\|_{L^{2}}(\|a\|_{L^{\infty}}\|\Lambda^{s+\beta}u\|_{L^{2}}+\|\Lambda^{1+\beta} u\|_{L^{2}}\|\Lambda^{s}a\|_{L^{2}}).
\end{align*}
By inserting the estimates from $I_4$ to $I_9$ into (\ref{307}), we find
\begin{align}\label{308}
&\frac{d}{dt}\langle\Lambda^{s-\beta} \nabla a,k\Lambda^{s-\beta}u\rangle+k\gamma\|\Lambda^{s-\beta}\nabla a\|_{L^2}^2\\\nonumber
\le\ &C\delta D_{0}+Ck\|\Lambda^{\beta}u\|_{H^{s}}^{2}+\frac{k}{100}\|\Lambda^{s-\beta+1}a\|_{L^2}.
\end{align}
Combining (\ref{303}), (\ref{305}), (\ref{306}) and (\ref{308}), we conclude that for any $k>0$
\begin{align*}
&\frac{1}{2}\frac{d}{dt}\left(\|(\sqrt{\gamma} a,u)\|_{H^{s}}^{2}+2k\langle\Lambda^{\beta-1}\nabla a,\Lambda^{\beta-1}u\rangle_{H^{s-2\beta+1}}\right)\\
&+\left(1-Ck\right)\|\Lambda^{\beta}u\|_{H^{s}}^{2}+\left(k\gamma-\frac{k}{100}\right)\|\Lambda^{\beta}u\|_{H^{s-2\beta+1}}^{2}\\
&\lesssim  \delta D_0.
\end{align*}
Choosing $k$ small enough, we obtain
\begin{align}\label{309}
\frac{d}{dt}E_0+D_0\le 0.
\end{align}
By $\frac{1}{2}\le\beta<1$, we have 
\begin{align*}
2k\langle\Lambda^{\beta-1}\nabla a,\Lambda^{\beta-1}u\rangle_{H^{s-2\beta+1}}\le &\ 2k\|a\|_{H^{s}}\|u\|_{H^{s+1-2\beta}}\\ \nonumber
\le &\ 2k\|a\|_{H^{s}}\|u\|_{H^{{+s}}}\\ \nonumber
\le &\ \frac{1}{2}\|a\|_{H^{s}}^{2}+2k^{2}\|u\|_{H^{{s}}}^{2}.\nonumber
\end{align*}
Taking $k$ small enough and integrating (\ref{309}) in time on $[0,t]$, then we conclude that
\begin{align*}
\sup_{t}\|(a,u)\|_{H^{s}}^{2}+\int_{0}^{t}D_{0}(t^{\prime})dt^{\prime}\lesssim \|(a_0,u_0)\|_{H^{s}}^{2}\le \frac{\delta^2}{4}
\end{align*}
holds for small enough $\delta>0$. Then we achieve the conclusion.\hfill$\Box$

\section{The time decay rate of the global small solution to (\ref{eq1})}
In this section, we investigate the optimal time decay rate for (\ref{eq1}). Since we focus on the long-time behavior of solutions, $t$ will be taken to be sufficiently large. Moreover, we agree that all occurrences of $\delta$ and $C_k$ denote their positive powers throughout this paper.
\begin{prop}{\label{4prop2}}
  Under the same conditions as in Theorem \ref{1theo1}, if additionally $\left(a_{0}, u_{0}\right) \in \dot{B}_{2, \infty}^{-1}$, then there exists $C>0$ such that for every $t>0$, there holds
\begin{align*}
\|\Lambda^{s_1}(a,u)\|_{H^{s-s_{1}}}\le C(1+t)^{-\frac{1+s_{1}}{2\beta}},
\end{align*}
where $0\le s_{1}\le\beta$.
\end{prop}
\begin{proof}
  Denote $S_{1}(t)=\left\{\xi:|\xi|^{2\beta}\le C_{2}(1+t)^{-1}\right\}$, $C_{2}$ is large enough. By Theorem \ref{1theo1}, we get
  \begin{align}\label{402}
  \frac{d}{dt}E_{0}(t)+\frac{C_{2}}{1+t}\left(k\gamma\|a\|_{H^{s}}^{2}+\|u\|_{H^{s}}^{2}\right)\lesssim \frac{1}{1+t}{\int_{S_{1}(t)}|\hat{a}(\xi)|^{2}+|\hat u}(\xi)|^{2}d\xi.
  \end{align}  
Applying
Fourier transform to (\ref{eq1}), we have
\begin{align*}
\left\{\begin{array}{l}\widehat{a}_{t}-i\xi_{k}\widehat{u}^{k}=\widehat{F},\\ [1ex]
 \widehat{u}^{j}_{t}+|\xi|^{2\beta}\widehat{u}^{j}+i\gamma\xi_{j}\widehat{a}=\widehat{H}^{j},\end{array}\right.
\end{align*}
where $F=-\mathrm{div}(au)$ and $H=K(a)\nabla a-u\cdot\nabla u+\frac{a}{1+a}(-\Delta)^{\beta}u$,
Then we deduce
\begin{align}\label{403}
\frac{1}{2}\frac{d}{dt}(\gamma|\widehat{a}|^{2}+|\widehat{u}|^{2})+|\xi|^{2\beta}|\widehat{u}|^{2}=\gamma\mathcal{R}e[\widehat{F} \cdot\overline{\widehat{a}}]+\mathcal{R}e[\widehat{H}\cdot\overline{\widehat{u}}]
\end{align}
Integrating (\ref{403}) in time on $[0,t]$, we get
\begin{align}\label{404}
\gamma|\widehat{a}|^{2}+|\widehat{u}|^{2}\lesssim (|\widehat{a}_{0}|^{2}+|\widehat{u}_{0}|^{2})+\int_{0}^{t}|\widehat{F}\cdot\overline{\widehat{a}}|+|\widehat{H}\cdot\overline{\widehat{u}}|dt^{\prime}.
\end{align}
Due to the fact $E(0)<\infty$ and $(a_0,u_0)\in \dot B_{2,\infty}^{-1}$, we find
\begin{align}\label{405}
\int_{S_{1}(t)}(|\widehat{a_{0}}|^{2}+|\widehat{u_{0}}|^{2})d\xi &\lesssim\sum_{j\leq\log_{2}\left[\frac{4}{3}C_{2}^{\frac{1}{2\beta}}(1+t)^{-\frac{1}{2\beta}}\right]}\int_{\mathbb{R}^{2}}2\varphi^{2}(2^{-j}\xi)\left(|\widehat{a_{0}}|^{2}+|\widehat{u_{0}}|^{2}\right)d\xi \\ \nonumber
 & \lesssim\sum_{j\leq\log_{2}\left[\frac{4}{3}C_{2}^{\frac{1}{2\beta}}(1+t)^{-\frac{1}{2\beta}}\right]}\left(\|\dot{\Delta}_{j}a_{0}\|_{L^{2}}^{2}+\|\dot{\Delta}_{j}u_{0}\|_{L^{2}}^{2}\right) \\ \nonumber
 & \lesssim (1+t)^{-\frac{1}{\beta}}\|(a_{0},u_{0})\|_{\dot{B}_{2,\infty}^{-1}}^{2}.\nonumber
\end{align}
By Minkowski's inequality, we arrive at
\begin{align}\label{406}
\int_{S_{1}(t)}\int_{0}^{t}|\widehat{F}\cdot\overline{\widehat{a}}|+|\widehat{H}\cdot\overline{\widehat{u}}|dt^{\prime}d\xi=& \int_{0}^{t}\int_{S_{1}(t)}|\widehat{F}\cdot\overline{\widehat{a}}|+|\widehat{H}\cdot\overline{\widehat{u}}|d\xi dt^{\prime} \\\nonumber
\lesssim &\ |S_{1}(t)|^{\frac{1}{2}}\int_{0}^{t}\|F\|_{L^1}\|a\|_{L^2}+\|H\|_{L^1}\|u\|_{L^2}dt^{\prime}\\\nonumber
\lesssim &\ (1+t)^{-\frac{1}{2\beta}}\int_{0}^{t}\left(\|a\|_{L^2}^2+\|u\|_{L^2}^2\right)\left(\|\nabla a\|_{L^2}+\|\nabla u\|_{L^2}+\|\Lambda^{2\beta}u\|_{L^2}\right)dt^{\prime}\\\nonumber
\lesssim &\ (1+t)^{-\frac{1}{2\beta}+\frac{1}{2}}\left(\int_{0}^{t}D_{0}(s)dt^{\prime}\right)^{\frac{1}{2}}\\\nonumber
\lesssim &\ (1+t)^{-\frac{1}{2\beta}+\frac{1}{2}}.\nonumber
\end{align}
Inserting (\ref{405}) and (\ref{406}) into (\ref{404}), we have
\begin{align}\label{407}
{\int_{S_{1}(t)}|\widehat{a}(\xi)|^{2}+|\widehat u}(\xi)|^{2}d\xi\lesssim (1+t)^{-\frac{1}{2\beta}+\frac{1}{2}}.
\end{align}
Plugging (\ref{407}) into (\ref{402}) gives rise to 
\begin{align*}
  \frac{d}{dt}E_{0}(t)+\frac{C_{2}}{1+t}\left(k\gamma\|a\|_{H^{s}}^{2}+\|u\|_{H^{s}}^{2}\right)\lesssim (1+t)^{-(\frac{1}{2\beta}-\frac{1}{2})}.
  \end{align*}  
Consequently, we get the initial time decay rate
\begin{align}\label{408}
E_{0}(t)\lesssim (1+t)^{-\frac{1}{2\beta}+\frac{1}{2}}.
  \end{align}
Plugging (\ref{406}), (\ref{407}) into (\ref{402}), we achieve
\begin{align*}
  \frac{d}{dt}E_{0}(t)&+\frac{C_{2}}{1+t}\left(k\gamma\|a\|_{H^{s}}^{2}+\|u\|_{H^{s}}^{2}\right)\lesssim (1+t)^{-1-\frac{1}{\beta}}\\\nonumber
  &+(1+t)^{-1-\frac{1}{2\beta}}\int_{0}^{t}\left(\|a\|_{L^2}^2+\|u\|_{L^2}^2\right)\left(\|\nabla a\|_{L^2}+\|\nabla u \|_{L^2}+\|\Lambda^{2\beta}u\|_{L^2}\right)dt^{\prime},\nonumber
  \end{align*} 
which implies
\begin{align}\label{409}
(1+t)^{1+\frac{1}{2\beta}}\frac{d}{dt}E_{0}(t)&+C_{2}(1+t)^{\frac{1}{2\beta}}\left(k\gamma\|a\|_{H^{s}}^{2}+\|u\|_{H^{s}}^{2}\right)\\\nonumber
&\lesssim (1+t)^{-\frac{1}{2\beta}}+\int_{0}^{t}\left(\|a\|_{L^2}^2+\|u\|_{L^2}^2\right)\left(\|\nabla a\|_{L^2}+\|\nabla u|_{L^2}+\|\Lambda^{2\beta}u\|_{L^2}\right)dt^{\prime}.\nonumber
\end{align}
Integrating (\ref{409}) in time on $[0,t]$, we get
\begin{align*}
(1+t)^{\frac{1}{2\beta}}E_{0}(t)\lesssim 1+\int_{0}^{t}\left(\|a\|_{L^2}^2+\|u\|_{L^2}^2\right)\left(\|\nabla a\|_{L^2}+\|\nabla u\|_{L^2}+\|\Lambda^{2\beta}u\|_{L^2}\right)dt^{\prime}.
\end{align*}
We set $\displaystyle N(t)=\sup_{0\le t^{\prime}\le t}(1+t^{\prime})^{\frac{1}{2\beta}}E_{0}(t^{\prime})$, then we find
\begin{align*}
N(t)\le C+C\int_{0}^{t}(1+t^{\prime})^{-\frac{1}{2\beta}}N(t^{\prime})(\|\nabla a\|_{L^2}+\|\nabla u\|_{L^2}+\|\Lambda^{2\beta}u\|_{L^2})dt^{\prime}.
\end{align*}
Applying Gronwall's inequality yields for any $t>0$, $N(t)<\infty$, which gives rise to
%\begin{align}\label{364}
%N(t)\lesssim \exp\left(\displaystyle\int_{0}^{t}(1+t^{\prime})^{-\frac{3}{4\beta}-\frac{1}{4}}dt^{\prime}\right)<+\infty.
%\end{align}
\begin{align}\label{410}
  E_{0}(t)\lesssim (1+t)^{-\frac{1}{2\beta}}.
\end{align}

Next we will prove the solution of (\ref{eq1}) belongs to some negative index Besov space. Applying $\dot\Delta_{j}$ to (\ref{eq1}), we find
\begin{align*}
 \left\{\begin{array}{l}\dot{\Delta}_{{j}}a_{t}+\text{div}\ \dot{\Delta}_{j}u=\dot{\Delta}_{j}F,\\ [1ex] \dot{\Delta}_{j}u_{t}+(-\Delta)^{\beta}\dot{\Delta}_{j}u+\gamma\dot{\Delta}_{j}\nabla a=\dot{\Delta}_{j}H.\end{array}\right.
\end{align*}
Then we get
\begin{align}\label{411}
\frac{d}{dt}\left(\gamma\|\dot{\Delta}_{j} a\|_{L^{2}}^{2}+\|\dot{\Delta}_{j}u\|_{L^{2} }^{2}\right)+2\|\Lambda^{\beta}\dot{\Delta}_{j}u\|_{L^{2}}^{2}\lesssim \|\dot{\Delta}_{j}F\| _{L^{2}}\|\dot{\Delta}_{j}a\|_{L^{2}}+\|\dot{\Delta}_{j}H\|_{L^{2}}\|\dot{ \Delta}_{j}u\|_{L^{2}}.
\end{align}
Multiplying (\ref{411}) by $2^{-2j}$ and taking $l^{\infty}$ norm, we have
\begin{align*}
\frac{d}{dt}\left(\gamma\|a\|_{\dot{B}^{-1}_{2,\infty}}^{2}+\|u\|_{\dot{B}^{-1}_{2,\infty}}^{2}\right)\lesssim \|F\|_{\dot{B}^{-1}_{2,\infty}}\|a\|_{ \dot{B}^{-1}_{2,\infty}}+\|H\|_{\dot{B}^{-1}_{2,\infty}}\|u\|_{ \dot{B}^{-1}_{2,\infty}}.  
\end{align*}
Let $\displaystyle M(t)=\sup_{0\le t^{\prime}\le t}\left(\gamma\|a\|_{\dot{B}^{-1}_{2,\infty}}+\|u\|_{\dot{B}^{-1}_{2,\infty}}\right)$, then we yield
\begin{align*}
M^{2}(t)\leq CM^{2}(0)+CM(t)\int_{0}^{t}\left(\|F\|_{\dot{B}^{-1}_{2,\infty}}+\|H\|_{\dot{B}^{-1}_{2,\infty}}\right)dt^{\prime}. 
\end{align*}
For any $\delta<\frac{1}{2\beta}$, due to (\ref{309}), we have
\begin{align*}
(1+t)^{\delta}\frac{d}{dt}{E}_0+(1+t)^{\delta}D_{0}\le 0.
\end{align*}
This together with (\ref{410}) ensures that
\begin{align*}
  (1+t)^{\delta}E_0+\int_{0}^{t}(1+t^{\prime})^{\delta}D_{0}dt^{\prime}\le C+\delta\int_{0}^{t}(1+t^{\prime})^{\delta-1}E_{0}dt^{\prime}\le C,
\end{align*}
Note that $ L^{1}\hookrightarrow\dot{B}^{-1}_{2,\infty}$, we conclude, for any $t> 0$ and $1-\frac{1}{2\beta}<\delta<\frac{1}{2\beta}$, that
\begin{align}\label{413}
\int_{0}^{t}(\|F\|_{\dot{B}^{-1}_{2,\infty}}+\|H\|_{\dot{B}^{-1}_{2,\infty}})dt^{\prime}&\lesssim  \int_{0}^{t}(\|F\|_{L^1}+\|H\|_{L^1})dt^{\prime}\\\nonumber
&\lesssim  \int_{0}^{t}(\|a\|_{L^2}+\|u\|_{L^2})(\|\nabla a\|_{L^2}+\|\nabla u\|_{L^2}+\|\Lambda^{2\beta} u\|_{L^2})dt^{\prime}\\\nonumber
&\lesssim  \int_{0}^{t}(1+t^{\prime})^{-\frac{1}{4\beta}}D_{0}^{\frac{1}{2}}dt^{\prime}\\\nonumber
&\lesssim  \int_{0}^{t}(1+t^{\prime})^{-\frac{1}{2\beta}-\delta}dt^{\prime}\int_{0}^{t}(1+t^{\prime})^{\delta}D_{0}^{2}dt^{\prime}<\infty.\nonumber
\end{align}
Hence, we deduce $M(t)<C$. From this we can obtain the optimal time decay rate for $E_{0}$. By (\ref{413}), we arrive at
\begin{align}\label{414}
\int_{S_{1}(t)}\int_{0}^{t}|\widehat{F}\cdot\overline{\widehat{a}}|+|\widehat{H}\cdot\overline{\widehat{u}}|dt^{\prime}d\xi=& \int_{0}^{t}\int_{S_{1}(t)}|\widehat{F}\cdot\overline{\widehat{a}}|+|\widehat{H}\cdot\overline{\widehat{u}}|d\xi dt^{\prime} \\\nonumber
\lesssim & \int_{0}^{t}\left(\|F\|_{L^1}\int_{S_{1}(t)}|{\widehat{a}}|d\xi+\|H\|_{L^1}\int_{S_{1}(t)}|{\widehat{u}}|d\xi\right
) dt^{\prime} \\\nonumber
\lesssim & \ (1+t)^{-\frac{1}{2\beta}}\int_{0}^{t}\left[\left(\|F\|_{L^1}+\|H\|_{L^1}\right)\left(\int_{S_{1}(t)}|\widehat{a}|^{2}+|\widehat{u}|^{2}d\xi\right)^{\frac{1}{2}}\right]dt^{\prime}\\\nonumber
\lesssim &\ (1+t)^{-\frac{1}{\beta}},\nonumber
\end{align}
where we use the fact
\begin{align*}
\left(\int_{S_{1}(t)}|\widehat{a}|^{2}+|\widehat{u}|^{2}d\xi\right)^{\frac{1}{2}}dt^{\prime}\lesssim (1+t)^{-\frac{1}{2\beta}}M(t).
\end{align*}
Thanks to (\ref{414}), (\ref{405}) and (\ref{402}), we achieve
\begin{align}\label{415}
 E_{0}(t)\lesssim  (1+t)^{-\frac{1}{\beta}}.
  \end{align}
From (\ref{309}) and (\ref{415}), we know
\begin{align}\label{416}
  \frac{d}{dt}\left[(1+t)^{\frac{1}{\beta}+1}E_{0}\right]+(1+t)^{\frac{1}{\beta}+1}D_{0}\lesssim  (1+t)^{\frac{1}{\beta}}E_{0},
\end{align}
Integrating (\ref{416}) over $[0,t]$, we infer that
\begin{align*}
  (1+t)^{\frac{1}{\beta}+1}E_{0}+\int_{0}^{t}(1+t^{\prime})^{\frac{1}{\beta}+1}D_{0}dt^{\prime}\lesssim&  \int_{0}^{t}(1+t^{\prime})^{\frac{1}{\beta}}E_{0}dt^{\prime}\lesssim C(1+t),
\end{align*}
which gives rise to
\begin{align}\label{417}
(1+t)^{-1}\int_{0}^{t}(1+t^{\prime})^{\frac{1}{\beta}+1}D_{0}dt^{\prime}\le C.
\end{align}
Next we want to study the $\beta$-order time decay rate, We first get, taking the $L^2$ inner product of the first two equations of (\ref{eq1}) with $(\gamma\Lambda^{\beta} a,\Lambda^{\beta} u)$, that 
 \begin{align}\label{418}
&\frac{1}{2}\frac{d}{dt}\|(\sqrt{\gamma}\Lambda^{\beta} a,\Lambda^{\beta} u)\|_{L^{2}}^{2}+\|\Lambda^{2\beta}u\|_{L^{2}}^{2}\\\nonumber
=&-\langle \Lambda^{\beta} (\mathrm{div}(au)),\Lambda^{\beta} a \rangle
-\langle \Lambda^{\beta} (k(a) \nabla a),\Lambda^{\beta} u\rangle
-\langle \Lambda^{\beta}(u\cdot\nabla u), \Lambda^{\beta} u\rangle\\\nonumber
&+\langle \Lambda^{\beta}\frac{a}{1+a}(-\Delta)^{\beta}u, \Lambda^{\beta} u\rangle\\\nonumber
\triangleq&\ J_1+J_2+J_3+J_4.\nonumber
\end{align}
For $J_1$, we have
\begin{align*}
  J_1=& -\langle \Lambda^{\beta} (\mathrm{div}(au)),\Lambda^{\beta} a \rangle\\
  =& -\langle \Lambda^{\beta} (u\cdot\nabla a),\Lambda^{\beta} a \rangle-\langle \Lambda^{\beta} (a\mathrm{div}u ),\Lambda^{\beta} a \rangle\\
  =& -\langle [\Lambda^{\beta}, u\cdot\nabla] a,\Lambda^{\beta} a \rangle-\langle u\cdot\nabla\Lambda^{\beta} a,\Lambda^{\beta} a \rangle-\langle \Lambda^{\beta} (a\mathrm{div}u ),\Lambda^{\beta} a \rangle\\
  \lesssim&\ \|\Lambda^{\beta} u\|_{L^{\frac{2}{1-\beta}}}\|\Lambda^{\beta} a\|_{L^{\frac{2}{\beta}}}\|\nabla a\|_{L^{2}}+\|\nabla u\|_{L^{\frac{1}{1-\beta}}}\|\Lambda^{\beta} a\|_{L^{\frac{2}{\beta}}}^{2}+\|\Lambda a\|_{L^{\frac{2}{\beta}}}\| a\|_{L^{\frac{2}{1-\beta}}}\|\Lambda^{2\beta} u\|_{L^{2}}\\
  &+\|\Lambda^{2\beta-1} a\|_{L^{\frac{2}{\beta}}}\|\mathrm{div}u\|_{L^{\frac{1}{1-\beta}}}\|\Lambda a\|_{L^{\frac{2}{\beta}}}\\
  \lesssim&\ \|\Lambda^{2\beta} u\|_{L^{2}}\|\nabla a\|_{L^{2}}^{2}+\|\Lambda^{2-\beta} a\|_{L^{2}}\|\Lambda^{\beta} a\|_{L^{2}}\|\Lambda^{2\beta} u\|_{L^{2}}+\|\Lambda^{\beta} a\|_{L^{2}}\|\Lambda^{2\beta}u\|_{L^{2}}\|\Lambda^{2-\beta} a\|_{L^{2}}\\
  \le & C\|\nabla a\|_{L^{2}}^{4}+\frac{1}{100}\|\Lambda^{2\beta} u\|_{L^{2}}^{2}+C\|\Lambda^{2-\beta} a\|_{L^{2}}^{2}\|\Lambda^{\beta}a\|_{L^{2}}^{2}.
\end{align*}  
Similarly, we have
\begin{align*}
  J_2= -\langle k(a) \nabla a,\Lambda^{2\beta} u\rangle
  \lesssim \|\Lambda^{2\beta} u\|_{L^{2}}\| a\|_{L^{\frac{2}{1-\beta}}}\|\nabla a\|_{L^{\frac{2}{\beta}}}
  \lesssim\|\Lambda^{2\beta} u\|_{L^{2}}\|\Lambda^{2-\beta} a\|_{L^{2}}\|\Lambda^{\beta} a\|_{L^{2}},
\end{align*}
and
\begin{align*}
  J_3+J_4\lesssim\delta\|\Lambda^{2\beta} u\|_{L^{2}}.
\end{align*}
Plugging the estimates of $J_1$ to $J_4$ into (\ref{418}) yields 
 \begin{align}\label{419}
\frac{1}{2}\frac{d}{dt}\|(\sqrt{\gamma}\Lambda^{\beta} a,\Lambda^{\beta} u)\|_{L^{2}}^{2}+\frac{3}{4}\|\Lambda^{2\beta}u\|_{L^{2}}^{2}\lesssim\|\nabla a\|_{L^{2}}^{4}+\|\Lambda^{2-\beta} a\|_{L^{2}}^{2}\|\Lambda^{\beta}a\|_{L^{2}}^{2}.
\end{align}
From (\ref{304}) and the estimates of $I_1$ to $I_4$, one can obtain
\begin{align}\label{420}
&\frac{1}{2}\frac{d}{dt}\|(\sqrt{\gamma}\Lambda^{s} a,\Lambda^{s}u)\|_{L^2}^{2}+\|\Lambda^{s+\beta}u\|_{L^2}^2\\\nonumber
\lesssim&\|\Lambda^s a\|_{L^{2}}^{2}\|\Lambda^{2-\beta} a\|_{L^{2}}^{2}+\|\nabla u\|_{L^{2}}^{2}\|\Lambda^{s-\beta+1} u\|_{L^{2}}^{2}+\frac{1}{100}\|\Lambda^{1+\beta}u\|_{L^{2}}^2.
\end{align}
By (\ref{419}) and (\ref{420}), we have
\begin{align}\label{421}
&\frac{1}{2}\frac{d}{dt}\|(\sqrt{\gamma}\Lambda^{\beta} a,\Lambda^{\beta} u)\|_{H^{s-\beta}}^{2}+\frac{1}{2}\|\Lambda^{2\beta}u\|_{H^{s-\beta}}^2\\\nonumber
\lesssim&\| \Lambda^{\beta}a\|_{H^{s-\beta}}^{2}\|\Lambda^{2-\beta} a\|_{H^{s-1}}^{2}+\|\nabla u\|_{L^{2}}^{2}\|\Lambda^{s-\beta+1} u\|_{L^{2}}^{2}+\|\nabla a\|_{L^{2}}^{4}.
\end{align}
Multiplying (\ref{421}) by $(1+t)^{\frac{1}{\beta}+2}$, one can arrive at
\begin{align}\label{4211}
&\frac{d}{dt}(1+t)^{\frac{1}{\beta}+2}\|(\sqrt{\gamma}\Lambda^{\beta} a,\Lambda^{\beta} u)\|_{H^{s-\beta}}^{2}+(1+t)^{\frac{1}{\beta}+2}\|\Lambda^{2\beta}u\|_{H^{s-\beta}}^2\\\nonumber
\lesssim&(1+t)^{2}D_{0}+(1+t)^{\frac{1}{\beta}+1}\|(\sqrt{\gamma}\Lambda^{\beta} a,\Lambda^{\beta} u)\|_{H^{s-\beta}}^{2}\\\nonumber
\lesssim&(1+t)^{\frac{1}{\beta}+1}D_{0},
\end{align}
Integrating (\ref{4211}) over $[0,t]$, we infer that
\begin{align*}
(1+t)^{\frac{1}{\beta}+2}\|(\sqrt{\gamma}\Lambda^{\beta} a,\Lambda^{\beta} u)\|_{H^{s-\beta}}^{2}
\lesssim\int_{0}^{t}(1+t^{\prime})^{\frac{1}{\beta}+1}D_{0}dt^{\prime}.
\end{align*}
Therefore, we get
\begin{align*}
(1+t)^{\frac{1}{\beta}+1}\|(\sqrt{\gamma}\Lambda^{\beta} a,\Lambda^{\beta} u)\|_{H^{s-\beta}}^{2}
\lesssim\frac{1}{1+t}\int_{0}^{t}(1+t^{\prime})^{\frac{1}{\beta}+1}D_{0}dt^{\prime}\le C.
\end{align*}
As a result, it comes out
\begin{align}\label{422}
\|(\sqrt{\gamma}\Lambda^{\beta} a,\Lambda^{\beta} u)\|_{H^{s-\beta}}^{2}
\lesssim (1+t)^{-\frac{1}{\beta}-1}.
\end{align}
This completes the proof of Proposition \ref{4prop2}.
\end{proof}
   Now we introduce the energy and dissipation functionals for $(a,u)$ as follows:
 \begin{align*}
  \widetilde{E}_{s}&=(1+t)^{b}\left\|\Lambda^{s}(\sqrt{\gamma}a, u)\right\|_{L^{2}}^{2}+k\langle\Lambda^{s-\beta} \nabla a,\nabla \Lambda^{s-\beta} u\rangle,\\
\widetilde{D}_{s}&=(1+t)^{b}\|\Lambda^{s+\beta} u\|_{L^{2}}^{2}+\frac{k\gamma}{2}\|\nabla \Lambda^{s-\beta} a\|_{L^{2}}^{2},
 \end{align*}
 where $b=2-\frac{1}{\beta}\in [0,1)$ and $k>0$ is a small enough constant. Next we will prove the time decay rate for the highest derivative of the solution to (\ref{eq1}).
   \begin{prop}\label{4prop3}
Under the same conditions as in Proposition \ref{4prop2}, then there exists $C>0$ such that for every  $t>0$, there holds
\begin{align*}
\left\|\Lambda^{s_1}(a, u)\right\|_{L^{2}} \leq C(1+t)^{-\frac{s_1+1}{2 \beta}},
\end{align*}  
where $s_1\in [0,s]$.
\end{prop}
\begin{proof}
  It is easy to check that
 \begin{align}\label{430}
&\frac{1}{2} \frac{d}{d t}\|\Lambda^{s}(\sqrt{\gamma}a, u)\|_{L^{2}}^{2}+\|\Lambda^{s+\beta} u\|_{L^{2}}^{2}=\left\langle\gamma\Lambda^{s} F, \Lambda^{s} a\right\rangle+\left\langle\Lambda^{s} H, \Lambda^{s} u\right\rangle,
\end{align}
and
\begin{align}\label{431}
&\frac{d}{d t}\langle\Lambda^{s-\beta} u,\nabla \Lambda^{s-\beta} a\rangle+\gamma\|\nabla \Lambda^{s-\beta} a\|_{L^{2}}^{2} \\\nonumber
&=\langle-\Lambda^{s-\beta}F +\mathrm{div}\Lambda^{s-\beta} u, \mathrm{div}\Lambda^{s-\beta} u\rangle+\langle\Lambda^{s-\beta}(H-\Lambda^{2\beta} u), \nabla \Lambda^{s-\beta} a\rangle.\nonumber
 \end{align}
 Combining (\ref{430}) and (\ref{431}) leads to
 \begin{align}\label{432}
 \frac{d}{dt}\widetilde{E}_{s}+2\widetilde{D}_{s}=&\ b(1+t)^{b-1}\|\Lambda^{s}(\sqrt{\gamma} a,u)\|_{L^{2}}^{2}+2(1+t)^ {b}(\langle\gamma\Lambda^{s}F,\Lambda^{s}a\rangle+\langle\Lambda^{s}H,\Lambda^{s}u\rangle)\\ \nonumber
 &+k\langle(\Lambda^{s-\beta}H-\Lambda^{s+\beta}u),\nabla \Lambda^{s-\beta}a\rangle-k\langle\Lambda^{s-\beta}(F-\text{div }u),\text{div }\Lambda^{s-\beta}u\rangle.\nonumber
 \end{align}
By Lemma \ref{lemma3}, we observe that
 \begin{align}\label{433}
\|\Lambda^{s}a\|_{L^2}^{2}\lesssim& \|a\|_{L^2}^{\frac{2(1-\beta)}{s-\beta+1}}\|\Lambda^{s-\beta+1}a\|_{L^2}^{\frac{2s}{s-\beta+1}}\\\nonumber
\lesssim& (1+t)^{-\frac{s+1}{\beta}+b-1}+(1+t)^{\frac{(s-\beta+1)(1-\beta)}{s\beta}}\|\Lambda^{s-\beta+1}a\|_{L^2}^{2}\\\nonumber
\lesssim& (1+t)^{-\frac{s+1}{\beta}+b-1}+(1+t)^{\frac{(2-\beta)(1-\beta)}{\beta}}\|\Lambda^{s-\beta+1}a\|_{L^2}^{2}.\nonumber
 \end{align}
Hence, thanks to (\ref{304}) and (\ref{422}), we obtain
 \begin{align*}
 &(1+t)^{b}\langle\gamma\Lambda^{s}F,\Lambda^{s}a\rangle\\\nonumber
 \lesssim &\ (1+t)^{b}(\|\Lambda^{s+\beta} u\|_{L^{2}}\|\Lambda^{s-\beta+1} a\|_{L^{2}}\|\nabla a\|_{L^{2}}+\|\Lambda^{s-\beta+1} a\|_{L^{2}}\|\Lambda^{s} a\|_{L^{2}}\|\Lambda^{1+\beta}u\|_{L^{2}}\\\nonumber
&+\|\Lambda^{s+\beta}u \|_{L^{2}}\|\Lambda^{s-\beta+1} a\|_{L^{2}}\|a\|_{L^{\infty}})\\\nonumber
 \lesssim&\ \delta \widetilde{D}_{s}+(1+t)^{b}\|\Lambda^{\beta} u\|_{L^{2}}\|\Lambda^{s-\beta+1} a\|_{L^{2}}\|\Lambda^{s} a\|_{L^{2}}\\ \nonumber
 \lesssim&\ \delta \widetilde{D}_{s}+(1+t)^{b}\|\Lambda^{\beta} u\|_{L^{2}}\|\Lambda^{s} a\|_{L^{2}}^{2}\\\nonumber
  \lesssim&\ \delta \widetilde{D}_{s}+(1+t)^{-\frac{s+1}{\beta}+b-1}.
 \end{align*}
 Along the same line, by the estimates from $I_1$ to $I_3$, we deduce that
  \begin{align*}
&(1+t)^{b}\left\langle\Lambda^{s} H, \Lambda^{s} u\right\rangle\\
\lesssim &(1+t)^{b}\|\Lambda^{s+\beta}u\|_{L^2}(\|\Lambda^{s-\beta+1}a\|_{L^2}\|a\|_{L^\infty}+\|\Lambda^{s}a\|_{L^{2}}\|\Lambda^{2-\beta} a\|_{L^{2}}\\
&+\|\Lambda^{s-\beta+1}u\|_{L^2}\|\nabla u\|_{L^2}+\|\Lambda^{s}a\|_{L^{2}}\|\Lambda^{1+\beta}u\|_{L^{2}}+\|\Lambda^{s+\beta}u\|_{L^{2}}\|a \|_{L^{\infty}})\\
\lesssim &\delta \widetilde{D}_{s}+(1+t)^{b}\|\Lambda^{s+\beta} u\|_{L^{2}}(\|\Lambda^{\beta} a\|_{L^{2}}\|\Lambda^{s} a\|_{L^{2}}+\|\Lambda^{s}a\|_{L^{2}}\|\Lambda^{\beta}u\|_{L^{2}})\\ \nonumber
\lesssim&\ \delta \widetilde{D}_{s}+(1+t)^{-\frac{s+1}{\beta}+b-1}.
\end{align*}
where we use the fact
 \begin{align}\label{434}
  \|\Lambda^{s-\beta+1} u\|_{L^{2}}^{2} \lesssim 
   (1+t)^{-\frac{s+1}{\beta}+b-1}+(1+t)^{b}\|\Lambda^{s+\beta}u\|_{L^{2}}^{2}.
 \end{align}
Similarly, one has
\begin{align*}
&k\langle\Lambda^{s-\beta}F +\mathrm{div}\Lambda^{s-\beta} u, \mathrm{div}\Lambda^{s-\beta} u\rangle+k\langle\Lambda^{s-\beta}(H+\Lambda^{2\beta} u), \nabla \Lambda^{s-\beta} a\rangle\\\nonumber
\le&  C( \delta  + k){\widetilde{D}}_{s}+\frac{k}{100}\|\Lambda^{s-\beta+1} a\|_{L^{2}}^{2} + C(1+t)^{-\frac{s+1}{\beta}+b-1}.
 \end{align*}
 Combining above, we get
 \begin{align}\label{435}
&\frac{d}{dt}\widetilde{E}_{s}+2\widetilde{D}_{ s}\le\ b(1+t)^{b-1}\|\Lambda^{s}(\gamma a,u)\|_{L^{2}}^{2}\\\nonumber
 &+C\left( {\delta+ k}\right) {\widetilde{D}}_{s } + C(1+t)^{-\frac{s+1}{\beta}+b-1}+\frac{k}{100}\|\Lambda^{s-\beta+1} a\|_{L^{2}}^{2}.
 \end{align}
 Recall that $S_{1}(t)=\left\{\xi:|\xi|^{2\beta}\le C_{2}(1+t)^{-1}\right\}$. For sufficiently large $C_2$, from (\ref{435}) we conclude that
\begin{align*}
\frac{d}{dt}{\widetilde{E}}_{s}+\frac{kC_{2}}{2}(1+t)^{b-1}\|\Lambda^{s}(a,u)\|_{L^2} \lesssim  {\left( 1 + t\right) }^{-\frac{s+1}{\beta}+b-1}.
\end{align*}
Applying the time weighted energy estimate, taking  $kC_2$ big enough, by Lemma \ref{lemma3} and (\ref{415}), then we obtain
\begin{align*}
  &(1+t)^{\frac{s+1}{\beta}-b+1}\widetilde{E}_{s}+\int_{0}^{t}\frac{kC_{2}}{2}(1+t^{\prime})^{\frac{s+1}{\beta}}\|\Lambda^{s}(a,u)\|_{L^{2}}^{2}dt^{\prime} \\
  \lesssim& (1+t)+\int_0^t(1+t^{\prime})^{\frac{s+1}{\beta}-b}|\langle\Lambda^{s-\beta}u,\nabla\Lambda^{s-\beta}a\rangle| dt^{\prime}+\int_{0}^{t}(1+t^{\prime})^{\frac{s+1}{\beta}}\|\Lambda^{s}(a,u)\|_{L^{2}}^{2}dt^{\prime} \\
  \lesssim & (1+t)+\int_0^t(1+t^{\prime})^{\frac{s+1}{\beta}}\|\Lambda^{s}(a,u)\|_{L^2}^2dt^{\prime}+\int_0^t(1+t^{\prime})^{\frac{1-s}{\beta}}\|(a,u)\|_{L^2}^2dt^{\prime} \\
  \lesssim& (1+t)+\int_{0}^{t}(1+t^{\prime})^{\frac{s+1}{\beta}}\|\Lambda^{s}(a,u)\|_{L^{2}}^{2}dt^{\prime}.
\end{align*}
By virtue of the time weighted energy estimate and (\ref{415}) again, we achieve
\begin{align*}
(1+t)^{\frac{s+1}{\beta}+1}\|\Lambda^{s}(a,u)\|_{L^{2}}^{2} 
&\lesssim (1+t)+k(1+t)^{\frac{s+1}{\beta}+1-b}|\langle\Lambda^{s-\beta}u,\nabla\Lambda^{s-\beta}a\rangle|\\
 & \lesssim (1+t)+k(1+t)^{\frac{s+1}{\beta}+1}\|\Lambda^{s}(a,u)\|_{L^2}^2+k(1+t)^{\frac{1-s}{\beta}+1}\|(a,u)\|_{L^2}^2 \\
 & \lesssim (1+t)+k(1+t)^{\frac{s+1}{\beta}+1}\|\Lambda^{s}(a,u)\|_{L^{2}}^{2}+Ck(1+t)^{-\frac{s}{\beta}+1},
\end{align*}
which gives rise to
\begin{align}\label{436}
\left\|\Lambda^{s}(a, u)\right\|_{L^{2}} \leq C(1+t)^{-\frac{s+1}{2 \beta}}.
\end{align}
By (\ref{415}) and (\ref{436}), we complete the proof of Proposition \ref{4prop3}.
\end{proof}
\begin{prop}\label{4prop4}
  Under the same conditions as in Proposition \ref{4prop3}, if additionally $\left(a_{0}, u_{0}\right) \in H^{s+\beta}$  and  $0<\left|\int_{\mathbb{R}^{2}}\left(a_{0}, u_{0}\right) d x\right|$, then there exists $C_{\beta} \leq C $ such that
\begin{align*}
\left\|\Lambda^{s_1}(a, u)\right\|_{L^{2}} \geq \frac{C_{\beta}}{2}(1+t)^{-\frac{s_1+1}{2 \beta}},
\end{align*}
where $s_1\in [0,s]$.
\end{prop}
\begin{proof}
We first consider the linear system:
\begin{align}\label{440}
\left\{\begin{array}{l}\partial_{t}a_{L}+\mathrm{div}\ u_{L}=0,\ \\ [1ex] \partial_{t}u_{L}+\gamma\nabla(a_{L})+(-\Delta)^{\beta}u_{L}=0,\\[1ex]
  a_{L}|_{t=0}=a_{0},\ u_{L}|_{t=0}=u_{0}.\end{array}\right.
\end{align}
According to Propositions \ref{4prop2}-\ref{4prop3}, one can arrive at $\|\Lambda^{s_1}(a_{L},u_{L})\|_{L^2}\le C(1+t)^{-\frac{1+s_1}{2\beta}}$ and $(a_{L},u_{L})\in L^{ \infty}([0,\infty),\dot{B}_{2,\infty}^{-1})$. Applying Fourier transform to (\ref{440}), we get
\begin{align}\label{441}
\left\{\begin{array}{l}\partial_{t}\widehat{a}_{L}-i\xi _{k}\widehat{u}_{L}^{k}=0,\\ [1ex] \partial_{t}\widehat{u}_{L}^{j}+|\xi|^{2\beta}\widehat{u}_{L}^{j}+\gamma i\xi_ {k}\widehat{a}_{L}=0.\end{array}\right.
\end{align}
Then we observe that
\begin{align*}
\frac{1}{2}\frac{d}{dt}\left[e^{2|\xi|^{2\beta}t}|(\sqrt{\gamma}\widehat{a}_{L},\widehat{u}_{L})|^{2}\right]-\gamma| \xi|^{2\beta}e^{2|\xi|^{2\beta}t}|\widehat{a}_{L}|^{2}=0,
\end{align*}
which implies that
\begin{align*}
|\xi|^{2s_{1}}|(\sqrt{\gamma}\widehat{a}_{L},\widehat{u}_{L})|^{2}=|\xi|^{2s_{1}}e^{-2|\xi|^{2\beta}t }|(\sqrt{\gamma}\widehat{a}_{0},\widehat{u}_{0})|^{2}+\int_{0}^{t}2\gamma|\xi|^{2(s_{1}+\beta)}e^{-2|\xi|^{2\beta}(t-t^{\prime})}|\widehat{a}_{L}|^{2}dt^{\prime}.
\end{align*}
Due to the fact $0<c_{0}=|\int_{\mathbb{R}^{2}}(a_{0}, u_{0})dx|=|(\widehat{a}_{0}(0),\widehat{u}_{0}(0))|$, we deduce that there exists $\eta>0$ such that $|(\widehat{a}_{0}(\xi),\widehat{u}_{0}(\xi))|\geq\frac{c_{0}}{2}$ if $\xi\in B(0, \eta)$. Then we have 
 \begin{align}\label{442}
   \|(\sqrt{\gamma}a_{L},u_{L})\|_{\dot{H}^{s_{1}}}^{2}& \geq\int_{|\xi|\leq\eta}|\xi|^{2s_{1}}e^{-2|\xi|^{2\beta}t}|(\sqrt{\gamma}\widehat{a}_{0}, \widehat{u}_{0})|^{2}d\xi\\ \nonumber &\geq\frac{c_{0}^{2}}{4}\int_{|\xi|\leq\eta}|\xi|^{2s_{1}}e^{-2|\xi|^{2\beta}t}d\xi\\ \nonumber
   &\geq C_{\beta}^{2}(1+t)^{-\frac{1+s_{1}}{\beta}},\nonumber
 \end{align}
 where $C_\beta^2=\frac{c_{0}^{2}}{4}\int_{|y|\le \eta}|y|^{2s_{1}}e^{-2|y|^{2\beta}}dy$.
 
 Taking $a_{N}=a-a_{L}$ and $u_{N}= u-u_{L}$, by Propositions \ref{4prop2}-\ref{4prop3}, we easily deduce $\|\Lambda^{s_{1}}(a_{N},u_{N})\|_{L^{2}}^{2} \leq C(1+t)^{-\frac{s_{1}+1}{\beta}}$ and $(a_{N},u_{N})\in L^{ \infty}([0,\infty),\dot{B}_{2,\infty}^{-1})$,  Moreover, we have
 \begin{align}\label{443}
  \left\{\begin{array}{l}\partial_{t}a_{N}+\mathrm{div}\ u_{N}=F,\  \\ [1ex] \partial_{t}u_{N}+\gamma\nabla a_{N}+(-\Delta)^{\beta}u_{N}=H, \\[1ex]
   a_{N}|_{t=0}=u_{N}|_{t=0}=0.\end{array}\right. 
 \end{align}
 According to the time decay rates for $(a_N,u_N)$ and $(a,u)$, we conclude from (\ref{443}) that
\begin{align*}
\frac{1}{2}\frac{d}{dt}\|(\sqrt{\gamma}a_{N},u_{N})\|_{L^{2}}^ {2}+\|\Lambda^{\beta}u_{N}\|_{L^{2}}^{2}&=\gamma\langle F,a_{N}\rangle+ \langle H,u_{N}\rangle\\\nonumber
 &\lesssim \|(a_{N},u_ {N}\|_{L^{4}}\|\nabla(\tau,u)\|_{L^{2}}\|(a,u)\|_{L^{4}}+\|a\|_{L^{\infty}}\|u_{N}\|_{L^{2}}\|\Lambda^{2\beta}u\|_{L^{2}}\\ \nonumber
&\le C\delta(1+t)^{-\frac{1}{\beta}-1}+\|a\|_{L^{\infty}}\|u_{N}\|_{L^{2}}\|\Lambda^{2\beta}u\|_{L^{2}},\nonumber
\end{align*}
and
\begin{align*}
\frac{d}{dt}\langle\Lambda^{2\beta-2}u_{N},\nabla a_{N}\rangle+ \gamma\|\nabla\Lambda^{\beta-1} a_{N}\|_{L^{2}}^{2}=&\langle\Lambda^{2\beta-2}( H-(-\Delta)^{\beta}u_{N}),\nabla a_{N}\rangle\\\nonumber
&-\langle\Lambda^{2\beta-2}\ (F-\text{div }u_{N}),\text{div }u_{N}\rangle\\\nonumber
\lesssim &(1+t)^{-\frac{1}{\beta}-1}+\|\Lambda^{3\beta-1}u_{N}\|_{L^2}^{2},\nonumber
\end{align*}
which lead to
\begin{align*}
&\frac{d}{dt}[\|(\sqrt{\gamma}a_{N},u_{N})\|_{L^{2}}^ {2}+k_{0}\langle\Lambda^{2\beta-2}u_{N},\nabla a_{N}\rangle]+2\|\Lambda^{\beta}u_{N}\|_{L^{2}}^{2}+ k_{0}\gamma\|\nabla^{\beta} a_{N}\|_{L^{2}}^{2}\\
&\le\ C(\delta+k_{0})(1+t)^{-1-\frac{1}{\beta}}+k_0\|\Lambda^{3\beta-1}u_{N}\|_{L^2}^{2}.
\end{align*}
Then we conclude that
\begin{align}\label{444}
&\frac{d}{dt}[\|(\sqrt{\gamma}a_{N},u_{N})\|_{L^{2}}^{2}+2k_{0}\langle\Lambda^{2\beta-2}u_{N},\nabla a_{N}\rangle]+\frac{ k_{0} C_{2}}{1+t}\|a_{N}\|_{L^{2}}^{2}+\frac{2C_{2}}{1+t}\|u_{N}\|_{L^{2}}^{2} \\\nonumber
 \lesssim &\frac{C_{2}}{1+t}\int_{S_{1}(t)}|\widehat{a_{N}}(\xi)|^{2}+|\widehat{ u_{N}}(\xi)|^{2}d\xi+(\delta+k_{0})(1+t)^{-\frac{1}{\beta}-1}+k_0\|\Lambda^{3\beta-1}u_{N}\|_{L^2}^{2}.\nonumber
\end{align}
We deduce from a similar derivation of (\ref{406}) and (\ref{413}) that
\begin{align}\label{445}
\int_{S_{1}(t)}|\widehat{a_{N}}|^{2}+|\widehat{u_{N}}|^{2}d \xi&\lesssim \int_{S_{1}(t)}\int_{0}^{t}|\widehat{F}\cdot\overline{\widehat{a_ {N}}}|+|\widehat{H}\cdot\overline{\widehat{u_{N}}}|dt^{\prime}d\xi\\ \nonumber
&\lesssim (1+t)^{-\frac{1}{\beta}}\int_{0}^{t}\|(u,\tau)\|_{L^{2}}\left(\|\nabla a\|_{L^{2}}+\|\nabla u\|_{L^{2}}+\|\Lambda^{2\beta} u\|_{L^{2}}\right)\|(u_{N},\tau_{N})\|_{\dot{B}^{-1}_{2,\infty}}dt^{\prime}\\ &\lesssim \delta(1+t)^{-\frac{1}{\beta}}.\nonumber
\end{align}
Plugging (\ref{445}) into (\ref{444}), we find
\begin{align}\label{446}
&\frac{d}{dt}\left[\|(\sqrt{\gamma}a_{N},u_{N})\|_{L^{2}}^{2}+2k_{0}\langle\Lambda^{2\beta-2}u_{N},\nabla a_{N}\rangle\right]+\frac{ k_{0}C_{2}}{(1+t)}\|a_{N}\|_{L^{2}}^{2}+\frac{2C_{2}}{1+t}\|u_{N}\|_{L^{2}}^{2} \\ \nonumber
\lesssim& \left(\delta C_2+k_{0}\right)(1+t)^{-\frac{1}{\beta}-1}+k_0\|\Lambda^{3\beta-1}u_{N}\|_{L^2}^{2}.\nonumber
\end{align}
Thanks to Propositions \ref{4prop2}-\ref{4prop3} and (\ref{417}), integrating (\ref{446}) over $[0,t]$, we arrive at
\begin{align*}
&(1+t)^{1+\frac{1}{\beta}}\|(\sqrt{\gamma}a_{N},u_{N})\|_{L^{2}}^{2}\\\nonumber
\lesssim &\ \int_{0}^{t}2k_{0}(1+t^{\prime})^{\frac{1}{\beta}}\langle\Lambda^{2\beta-2}u_{N},\nabla a_{N}\rangle dt^{\prime}
+2k_{0}(1+t)^{1+\frac{1}{\beta}}\langle\Lambda^{2\beta-2}u_{N},\nabla a_{N}\rangle+(\delta C_2+k_{0})(1+t)\\\nonumber
&+k_0\int_{0}^{t}(1+t^{\prime})^{1+\frac{1}{\beta}}\|\Lambda^{3\beta-1}u_{N}\|_{L^2}^{2}dt^{\prime}\\
 \lesssim &\ \left(\delta C_2+k_{0}\right)(1+t),
\end{align*}
which yields
\begin{align}\label{447}
\|(\sqrt{\gamma}a_{N},u_{N})\|_{L^{2}}^{2}\le C(\delta C_2+k_{0})(1+t)^{-\frac{1}{\beta}}.
\end{align}
Applying $\Lambda^{s_{1}}$ to (\ref{443}), we get
\begin{align*}
\left\{\begin{array}{l}\partial_{t}\Lambda^{s_{1}}a_{N}+ \mathrm{div} \Lambda^{s_{1}}u_{N}=\Lambda^{s_{1}}F,\\ [1ex] \partial_{t}\Lambda^{s_{1}}u_{N}+\gamma \nabla\Lambda^{s_{1}}a_{N}+(-\Delta)^{\beta}\Lambda^{s_{1}}u_{N}=\Lambda^{s_{1}}H.\end{array}\right.
\end{align*}
Standard energy estimate yields 
 \begin{align}\label{448}
&\frac{1}{2} \frac{d}{d t}\left[(1+t)^{b}\|\Lambda^{s}(\sqrt{\gamma}a_N, u_N)\|_{L^{2}}^{2}\right]+(1+t)^{b}\|\Lambda^{s+\beta} u_N\|_{L^{2}}^{2}\\\nonumber
=&b(1+t)^{b-1}\|\Lambda^{s}(\sqrt{\gamma}a_N, u_N)\|_{L^{2}}^{2}+(1+t)^{b}\left\langle\gamma\Lambda^{s} F, \Lambda^{s} a_N\right\rangle+(1+t)^{b}\left\langle\Lambda^{s} H, \Lambda^{s} u_N\right\rangle.\nonumber
\end{align}
 Using Lemma \ref{lemma4}, Lemma \ref{lemma5} and Proposition \ref{4prop3}, we infer that
 \begin{align*}
&(1+t)^{b}\gamma\left\langle\Lambda^{s} F, \Lambda^{s} a_N\right\rangle \\
=&(1+t)^{b}(-\gamma\langle[\Lambda^{s},u\cdot\nabla]a,\Lambda^{s}a_N\rangle-\gamma\langle u\cdot\nabla \Lambda^{s}a,\Lambda^{s}a_{N}\rangle-\gamma\langle\Lambda^{s+\beta-1}(a\mathrm{div}u),\Lambda^{s-\beta+1}a_N\rangle)\\\nonumber
=&(1+t)^{b}(-\gamma\langle[\Lambda^{s},u\cdot\nabla]a,\Lambda^{s}a_N\rangle+\frac{\gamma}{2}\langle \mathrm{div}u, |\Lambda^{s}a_{N}|^{2}\rangle-\gamma\langle u\cdot\nabla \Lambda^{s}a_{L},\Lambda^{s}a_{N}\rangle\\
&-\gamma\langle\Lambda^{s+\beta-1}(a\mathrm{div}u),\Lambda^{s-\beta+1}a_N\rangle)\\\nonumber
\lesssim &(1+t)^{b}\left(\ \|\Lambda^{s+\beta} u\|_{L^{2}}\|\Lambda^{s-\beta+1} a_N\|_{L^{2}}\|\nabla a\|_{L^{2}}+\|\Lambda^{s-\beta+1} a_N\|_{L^{2}}\|\Lambda^{s} a_N\|_{L^{2}}\|\Lambda^{1+\beta}u\|_{L^{2}}\right.\\\nonumber
&\left.+\|\Lambda^{s+\beta}u \|_{L^{2}}\|\Lambda^{s-\beta+1} a_N\|_{L^{2}}\|a\|_{L^{\infty}}
+\delta(1+t)^{-\frac{s+1}{\beta}-1}\right)\\\nonumber
\lesssim & \delta(1+t)^{b}\|\Lambda^{s+\beta}u_{N}\|_{L^{2}}^{2}+\delta\|\Lambda^{s-\beta+1}a_{N}\|_{L^{2}}^{2}+\delta(1+t)^{-\frac{s+1}{\beta}+b-1}.
 \end{align*}
 Along the same line, we have
 \begin{align*}
 &(1+t)^{b}\left\langle\Lambda^{s} H, \Lambda^{s} u_N\right\rangle\\
 \lesssim &(1+t)^{b}\|\Lambda^{s+\beta}u_N\|_{L^2}(\|\Lambda^{s-\beta+1}a\|_{L^2}\|a\|_{L^\infty}+\|\Lambda^{s}a\|_{L^{2}}\|\Lambda^{2-\beta} a\|_{L^{2}}\\
&+\|\Lambda^{s-\beta+1}u\|_{L^2}\|\nabla u\|_{L^2}+\|\Lambda^{s}a\|_{L^{2}}\|\Lambda^{1+\beta}u\|_{L^{2}}+\|\Lambda^{s+\beta}u\|_{L^{2}}\|a \|_{L^{\infty}})\\
\lesssim&\ \delta(1+t)^{b}\|\Lambda^{s+\beta}u_{N}\|_{L^{2}}^{2}+\delta\|\Lambda^{s-\beta+1}a_{N}\|_{L^{2}}^{2}+\delta(1+t)^{-\frac{s+1}{\beta}+b-1}.
 \end{align*}
 Substituting the above inequalities into (\ref{448}) gives rise to
 \begin{align}\label{449}
&\frac{1}{2} \frac{d}{d t}\left[(1+t)^{b}\|\Lambda^{s}(\sqrt{\gamma}a_N, u_N)\|_{L^{2}}^{2}\right]+(1+t)^{b}\|\Lambda^{s+\beta} u_N\|_{L^{2}}^{2}\\\nonumber
\le &\ b(1+t)^{b-1}\|\Lambda^{s}(\sqrt{\gamma}a_N, u_N)\|_{L^{2}}^{2}+C\delta(1+t)^{b}\|\Lambda^{s+\beta}u_{N}\|_{L^{2}}^{2}+C\delta\|\Lambda^{s-\beta+1}a_{N}\|_{L^{2}}^{2}\\\nonumber
&+C\delta(1+t)^{-\frac{s+1}{\beta}+b-1}.\nonumber
\end{align}
Along the same line to the proof of (\ref{449}), we achieve
\begin{align}\label{450}
&\frac{d}{d t}\langle\Lambda^{s-\beta} u_N,\nabla \Lambda^{s-\beta} a_N\rangle+\gamma\|\nabla \Lambda^{s-\beta} a_N\|_{L^{2}}^{2} \\\nonumber
=&\langle-\Lambda^{s-\beta}F +\mathrm{div}\Lambda^{s-\beta} u_N, \mathrm{div}\Lambda^{s-\beta} u_N\rangle+\langle\Lambda^{s-\beta}(H-\Lambda^{2\beta} u_N), \nabla \Lambda^{s-\beta} a_N\rangle\\\nonumber
\le &C\|\Lambda^{s+\beta}u_{N}\|_{L^{2}}^{2}+C\|\Lambda^{s-\beta+1}u_{N}\|_{L^{2}}^{2}+C\left(\delta+\frac{1}{100}\right)\|\Lambda^{s-\beta+1}a_{N}\|_{L^{2}}^{2}+C\delta(1+t)^{-\frac{s+1}{\beta}+b-1}.
 \end{align}
Thanks to (\ref{434}) again, combining (\ref{449}) and (\ref{450}), we get
 \begin{align}\label{451}
&\frac{d}{d t}\left[(1+t)^{b}\|\Lambda^{s}(\sqrt{\gamma}a_N, u_N)\|_{L^{2}}^{2}+2\langle\Lambda^{s-\beta} u_N,\nabla \Lambda^{s-\beta} a_N\rangle\right]\\\nonumber
&+(1+t)^{b}\|\Lambda^{s+\beta} u_N\|_{L^{2}}^{2}+\gamma\|\nabla \Lambda^{s-\beta} a_N\|_{L^{2}}^{2}\\\nonumber
\le &b(1+t)^{b-1}\|\Lambda^{s}(\sqrt{\gamma}a_N, u_N)\|_{L^{2}}^{2}+C\delta(1+t)^{-\frac{s+1}{\beta}+b-1}.\nonumber
 \end{align}
 It follows from (\ref{451}) that
 \begin{align}\label{452}
&\frac{d}{d t}\left[(1+t)^{b}\|\Lambda^{s}(\sqrt{\gamma}a_N, u_N)\|_{L^{2}}^{2}+2\langle\Lambda^{s-\beta} u_N,\nabla \Lambda^{s-\beta} a_N\rangle\right]+\frac{C_2}{2}(1+t)^{b-1}\|\Lambda^{s} (a_N,u_N)\|_{L^{2}}^{2}\\\nonumber
\le &C_{2}(1+t)^{b-1}\int_{S_{1}(t)}|\xi|^{2s}\left(|\widehat{a_{N}}(\xi)|^{2}+|\widehat{u_{N}}(\xi)|^{2}\right)d \xi+C\delta(1+t)^{-\frac{s+1}{\beta}+b-1}\\\nonumber
\lesssim& \left[\delta+(\delta C_{2}+k_{0})C_{2}\right](1+t)^{-\frac{s+1}{\beta}+b-1}.
 \end{align}
Denote that 
 \begin{align*}
  \widetilde{E}_{s}&=(1+t)^{b}\|\Lambda^{s}(\sqrt{\gamma}a_N, u_N)\|_{L^{2}}^{2}+2\langle\Lambda^{s-\beta} u_N,\nabla \Lambda^{s-\beta} a_N\rangle.
 \end{align*}
Taking $C_2$ big enough, by (\ref{452}), then we obtain
\begin{align*}
  &(1+t)^{\frac{s+1}{\beta}-b+1}\widetilde{E}_{s}+\frac{C_{2}}{2}\int_{0}^{t}(1+t^{\prime})^{\frac{s+1}{\beta}}\|\Lambda^{s}(\sqrt{\gamma}a_N,u_N)\|_{L^{2}}^{2}dt^{\prime} \\
  \lesssim &\int_0^t(1+t^{\prime})^{\frac{s+1}{\beta}-b}\widetilde{E}_{s}dt^{\prime}+\left[\delta+(\delta C_{2}+k_{0})C_{2}\right](1+t)\\
  %\lesssim &\int_0^t(1+t^{\prime})^{\frac{s+1}{\beta}-b}|\langle\Lambda^{s-\beta}u_N,\nabla\Lambda^{s-\beta}a_N\rangle| +(1+t^{\prime})^{\frac{s+1}{\beta}}\|\Lambda^{s}(\sqrt{\gamma}a_N,u_N)\|_{L^{2}}^{2}dt^{\prime}+\left[\delta+(\delta C_{2}+k_{0})C_{2}\right](1+t)\\
  \lesssim &\int_0^t(1+t^{\prime})^{\frac{s+1}{\beta}}\|\Lambda^{s}(a_N,u_N)\|_{L^2}^2dt^{\prime}+\int_0^t(1+t^{\prime})^{\frac{1-s}{\beta}}\|(\sqrt{\gamma}a_N,u_N)\|_{L^2}^2dt^{\prime} +\left[\delta+(\delta C_{2}+k_{0})C_{2}\right](1+t)\\
  \lesssim &\int_{0}^{t}(1+t^{\prime})^{\frac{s+1}{\beta}}\|\Lambda^{s}(a_N,u_N)\|_{L^{2}}^{2}dt^{\prime}+\left[\delta+(\delta C_{2}+k_{0})C_{2}\right](1+t).
\end{align*}
Moreover, we infer that
\begin{align*}
&(1+t)^{\frac{s+1}{\beta}+1}\|\Lambda^{s}(\sqrt{\gamma}a_N,u_N)\|_{L^{2}}^{2} \\
\le &\left[\delta+(\delta C_{2}+k_{0})C_{2}\right](1+t)-2(1+t)^{\frac{s+1}{\beta}+1-b}\langle\Lambda^{s-\beta}u_N,\nabla\Lambda^{s-\beta}a_N\rangle\\
\le &\left[\delta+(\delta C_{2}+k_{0})C_{2}\right](1+t)+\frac{1}{2}(1+t)^{\frac{s+1}{\beta}+1}\|\Lambda^{s}(a_N,u_N)\|_{L^2}^2+C(1+t)^{\frac{1-s}{\beta}+1}\|(a_N,u_N)\|_{L^2}^2 \\
 \le &\left[\delta+(\delta C_{2}+k_{0})C_{2}\right](1+t)+\frac{1}{2}(1+t)^{\frac{s+1}{\beta}+1}\|\Lambda^{s}(a_N,u_N)\|_{L^{2}}^{2}.
\end{align*}
which gives rise to 
\begin{align}\label{453}
\|\Lambda^{s}(\sqrt{\gamma}a_{N},u_{N})\|_{L^{2}}^{2}\le C[\delta+(\delta C_{2}+k_{0})C_{2}](1+t)^{-\frac{s+1}{\beta}}.
\end{align}
Taking $\delta+(\delta C_{2}+k_{0})C_{2}$ small enough, from (\ref{447}) and (\ref{453}), we arrive at
\begin{align}\label{455}
\|\Lambda^{s_1}(\sqrt{\gamma}a_{N},u_{N})\|_{L^{2}}^{2}\le \frac{C_{\beta}^{2}}{4}(1+t)^{-\frac{s_1+1}{\beta}}.
\end{align}
Combining (\ref{442}) and (\ref{455}), it comes out
\begin{align*}
\|\Lambda^{s_1}(\sqrt{\gamma}a,u)\|_{L^{2}}^{2}\ge \frac{C_{\beta}^{2}}{4}(1+t)^{-\frac{s_1+1}{\beta}}.
\end{align*}
Therefore, we conclude the proof of Proposition \ref{4prop4}.
\end{proof}
\textbf{Proof of Theorem \ref{1theo2}:}
\begin{proof}
  Combining the proof of Proposition \ref{4prop3} and Proposition \ref{4prop4}, we finish the proof of Theorem \ref{1theo2}.
\end{proof}
\bigskip

 \vskip .1in
 \noindent{\bf Acknowledgement} This work was partially supported by the National Natural Science Foundation of
  China (No.12571261).

 \vskip .1in
\noindent{\bf Data Availability Statement} Data sharing is not applicable to this article as no
data sets were generated or analysed during the current study.

\vskip .1in

\noindent{\bf Conflict of Interest} The authors declare that they have no conflict of interest. The
authors also declare that this manuscript has not been previously published, and
will not be submitted elsewhere before your decision.

\phantomsection
%\addcontentsline{toc}{section}{\refname}
\bibliographystyle{abbrv}%选择你要用的格式
\bibliography{Manuscript1209}% 将mybibfile改为你的bib文件名,注意不加后缀

@book {Bahouri2011,
    AUTHOR = {Bahouri, Hajer and Chemin, Jean-Yves and Danchin, Rapha\"{e}l},
     TITLE = {Fourier analysis and nonlinear partial differential equations},
    SERIES = {Grundlehren der Mathematischen Wissenschaften [Fundamental
              Principles of Mathematical Sciences]},
    VOLUME = {343},
 PUBLISHER = {Springer, Heidelberg},
      YEAR = {2011},
     PAGES = {xvi+523},
      ISBN = {978-3-642-16829-1},
   MRCLASS = {35-02 (35L72 35Q30 42-02 42B37 76B03 76D03 76N10)},
  MRNUMBER = {2768550},
MRREVIEWER = {Peter R. Massopust},
       DOI = {10.1007/978-3-642-16830-7},
       URL = {https://doi.org/10.1007/978-3-642-16830-7},
}

@article {Schonbek1985,
    AUTHOR = {Schonbek, Maria Elena},
     TITLE = {{$L^2$} decay for weak solutions of the {N}avier-{S}tokes
              equations},
   JOURNAL = {Arch. Rational Mech. Anal.},
  FJOURNAL = {Archive for Rational Mechanics and Analysis},
    VOLUME = {88},
      YEAR = {1985},
    NUMBER = {3},
     PAGES = {209--222},
      ISSN = {0003-9527},
   MRCLASS = {35Q10 (76D05)},
  MRNUMBER = {775190},
MRREVIEWER = {Yoshikazu Giga},
       DOI = {10.1007/BF00752111},
       URL = {http://dx.doi.org/10.1007/BF00752111},
}

@article {Schonbek1991,
    AUTHOR = {Schonbek, Maria E.},
     TITLE = {Lower bounds of rates of decay for solutions to the
              {N}avier-{S}tokes equations},
   JOURNAL = {J. Amer. Math. Soc.},
  FJOURNAL = {Journal of the American Mathematical Society},
    VOLUME = {4},
      YEAR = {1991},
    NUMBER = {3},
     PAGES = {423--449},
      ISSN = {0894-0347},
   MRCLASS = {35Q30 (35B40 76D05)},
  MRNUMBER = {1103459},
MRREVIEWER = {Michael Wiegner},
       DOI = {10.2307/2939262},
       URL = {http://dx.doi.org/10.2307/2939262},
}

@article {Li2011Large,
    AUTHOR = {Li, Hai Liang and Zhang, Ting},
     TITLE = {Large time behavior of isentropic compressible
              {N}avier-{S}tokes system in {$R^3$}},
   JOURNAL = {Math. Methods Appl. Sci.},
  FJOURNAL = {Mathematical Methods in the Applied Sciences},
    VOLUME = {34},
      YEAR = {2011},
    NUMBER = {6},
     PAGES = {670--682},
      ISSN = {0170-4214},
   MRCLASS = {35Q35 (35B40 76N10)},
  MRNUMBER = {2814719},
MRREVIEWER = {Fa-gui Liu},
       DOI = {10.1002/mma.1391},
       URL = {https://doi.org/10.1002/mma.1391},
}

@article {Xu2019,
    AUTHOR = {Xu, Jiang},
     TITLE = {A low-frequency assumption for optimal time-decay estimates to
              the compressible {N}avier-{S}tokes equations},
   JOURNAL = {Comm. Math. Phys.},
  FJOURNAL = {Communications in Mathematical Physics},
    VOLUME = {371},
      YEAR = {2019},
    NUMBER = {2},
     PAGES = {525--560},
      ISSN = {0010-3616},
   MRCLASS = {35Q35 (35B40)},
  MRNUMBER = {4019913},
MRREVIEWER = {Beno\^{\i}t P. Desjardins},
       DOI = {10.1007/s00220-019-03415-6},
       URL = {https://doi.org/10.1007/s00220-019-03415-6},
}

@article {1959On,
    AUTHOR = {Nirenberg, L.},
     TITLE = {On elliptic partial differential equations},
   JOURNAL = {Ann. Scuola Norm. Sup. Pisa Cl. Sci. (3)},
  FJOURNAL = {Annali della Scuola Normale Superiore di Pisa. Classe di
              Scienze. Serie III},
    VOLUME = {13},
      YEAR = {1959},
     PAGES = {115--162},
      ISSN = {0391-173X},
   MRCLASS = {35.00},
  MRNUMBER = {109940},
MRREVIEWER = {L. Garding},
}

@book {2004Feireisl,
	AUTHOR = {Feireisl, Eduard},
	TITLE = {Dynamics of viscous compressible fluids},
	SERIES = {Oxford Lecture Series in Mathematics and its Applications},
	VOLUME = {26},
	PUBLISHER = {Oxford University Press, Oxford},
	YEAR = {2004},
	PAGES = {xii+212},
	ISBN = {0-19-852838-8},
	MRCLASS = {76N10 (35Q35 76N15)},
	MRNUMBER = {2040667},
	MRREVIEWER = {Piotr Bogus\l aw Mucha},
}

@book {1998Lions,
	AUTHOR = {Lions, Pierre Louis},
	TITLE = {Mathematical topics in fluid mechanics. {V}ol. 2},
	SERIES = {Oxford Lecture Series in Mathematics and its Applications},
	VOLUME = {10},
	NOTE = {Compressible models,
	Oxford Science Publications},
	PUBLISHER = {The Clarendon Press, Oxford University Press, New York},
	YEAR = {1998},
	PAGES = {xiv+348},
	ISBN = {0-19-851488-3},
	MRCLASS = {76-02 (35-02 35Q30 76N10)},
	MRNUMBER = {1637634},
	MRREVIEWER = {Denis Serre},
}

@book {1996Lions,
	AUTHOR = {Lions, Pierre Louis},
	TITLE = {Mathematical topics in fluid mechanics. {V}ol. 1},
	SERIES = {Oxford Lecture Series in Mathematics and its Applications},
	VOLUME = {3},
	NOTE = {Incompressible models,
	Oxford Science Publications},
	PUBLISHER = {The Clarendon Press, Oxford University Press, New York},
	YEAR = {1996},
	PAGES = {xiv+237},
	ISBN = {0-19-851487-5},
	MRCLASS = {76-02 (35Q30 35Q35 76D05)},
	MRNUMBER = {1422251},
	MRREVIEWER = {Denis Serre},
}

@article{kato1,
author = {Kato, Tosio and Ponce, Gustavo},
title = {Commutator estimates and the euler and navier-stokes equations},
journal = {Commun. Pure Appl. Math.},
volume = {41},
number = {7},
pages = {891-907},
year = {1988},
}

@article{1979Matsumura,
  title={The initial value problem for the equations of motion of compressible viscous and heat-conductive fluids},
  author={ Matsumura, A  and  Nishida, T },
  journal={Proc. Jpn. Acad.,Ser. A,Math. Sci.},
  volume={55},
  number={9},
  doi = {https://doi.org/10.3792/pjaa.55.337},
  pages={17408},
  year={1979},
}

@article{1980Matsumura,
  title={The initial value problem for the equations of motion of viscous and heat-conductive gases},
  author={ Matsumura, Akitaka  and  Nishida, Takaaki },
  journal={J. Math. Kyoto Univ.},
doi={https://doi.org/10.1215/kjm/1250522322},
  volume={20},
  pages={67-104},
  year={1980},
}

@article{PONCE1985399,
title = {Global existence of small solutions to a class of nonlinear evolution equations},
journal = {Nonlinear Anal. Theory Methods Appl.},
volume = {9},
number = {5},
pages = {399-418},
year = {1985},
issn = {0362-546X},
doi = {https://doi.org/10.1016/0362-546X(85)90001-X},
url = {https://www.sciencedirect.com/science/article/pii/0362546X8590001X},
author = {Gustavo Ponce},
}

@article{xin1988blowup,
  title={Blowup of smooth solutions to the compressible {N}avier-{S}tokes equation with compact density},
  author={Xin, Zhouping},
  journal={Comm. Pure Appl. Math.},
  volume={51},
  number={3},
doi = {https://doi.org/10.1002/(SICI)1097-0312(199803)51:3<229::AID-CPA1>3.0.CO;2-C},
url = {https://onlinelibrary.wiley.com/doi/abs/10.1002/%28SICI%291097-0312%28199803%2951%3A3%3C229%3A%3AAID-CPA1%3E3.0.CO%3B2-C},
  pages={299-440},
  year={1988},
  publisher={Wiley Online Library},
}

@article {MR1779621,
    AUTHOR = {Danchin, R.},
     TITLE = {Global existence in critical spaces for compressible
              {N}avier-{S}tokes equations},
   JOURNAL = {Invent. Math.},
  FJOURNAL = {Inventiones Mathematicae},
    VOLUME = {141},
      YEAR = {2000},
    NUMBER = {3},
     PAGES = {579--614},
      ISSN = {0020-9910,1432-1297},
   MRCLASS = {76N10 (35Q30)},
  MRNUMBER = {1779621},
MRREVIEWER = {Kevin\ R.\ Zumbrun},
       DOI = {10.1007/s002220000078},
       URL = {https://doi.org/10.1007/s002220000078},
}

@article{Guo01012012,
author = {Yan Guo and Yanjin Wang},
title = {Decay of Dissipative Equations and Negative Sobolev Spaces},
journal = {Communications in Partial Differential Equations},
volume = {37},
number = {12},
pages = {2165--2208},
year = {2012},
publisher = {Taylor \& Francis},
doi = {10.1080/03605302.2012.696296},
URL = {https://doi.org/10.1080/03605302.2012.696296},
eprint = {  https://doi.org/10.1080/03605302.2012.696296}
}

@article {MR2675485,
    AUTHOR = {Chen, Qionglei and Miao, Changxing and Zhang, Zhifei},
     TITLE = {Global well-posedness for compressible {N}avier-{S}tokes
              equations with highly oscillating initial velocity},
   JOURNAL = {Comm. Pure Appl. Math.},
  FJOURNAL = {Communications on Pure and Applied Mathematics},
    VOLUME = {63},
      YEAR = {2010},
    NUMBER = {9},
     PAGES = {1173--1224},
      ISSN = {0010-3640,1097-0312},
   MRCLASS = {35B65 (35Q35 76N10)},
  MRNUMBER = {2675485},
MRREVIEWER = {Paolo\ Maremonti},
       DOI = {10.1002/cpa.20325},
       URL = {https://doi.org/10.1002/cpa.20325},
}

@article{DUAN2007220,
title = {Optimal {${L}^p$}–{${L}^q$} convergence rates for the compressible {N}avier–{S}tokes equations with potential force},
journal = {Journal of Differential Equations},
volume = {238},
number = {1},
pages = {220-233},
year = {2007},
issn = {0022-0396},
doi = {https://doi.org/10.1016/j.jde.2007.03.008},
url = {https://www.sciencedirect.com/science/article/pii/S0022039607000836},
author = {Renjun Duan and Hongxia Liu and Seiji Ukai and Tong Yang}
}

@article {MR4188989,
    AUTHOR = {Xin, Zhouping and Xu, Jiang},
     TITLE = {Optimal decay for the compressible {N}avier-{S}tokes equations
              without additional smallness assumptions},
   JOURNAL = {J. Differential Equations},
  FJOURNAL = {Journal of Differential Equations},
    VOLUME = {274},
      YEAR = {2021},
     PAGES = {543--575},
      ISSN = {0022-0396,1090-2732},
   MRCLASS = {76N06 (35K65 35L65 35Q30)},
  MRNUMBER = {4188989},
MRREVIEWER = {Piotr\ Biler},
       DOI = {10.1016/j.jde.2020.10.021},
       URL = {https://doi.org/10.1016/j.jde.2020.10.021},
}

@article {MR3609245,
    AUTHOR = {Danchin, Rapha\"el and Xu, Jiang},
     TITLE = {Optimal time-decay estimates for the compressible
              {N}avier-{S}tokes equations in the critical {$L^p$} framework},
   JOURNAL = {Arch. Ration. Mech. Anal.},
  FJOURNAL = {Archive for Rational Mechanics and Analysis},
    VOLUME = {224},
      YEAR = {2017},
    NUMBER = {1},
     PAGES = {53--90},
      ISSN = {0003-9527,1432-0673},
   MRCLASS = {35Q35 (35B40 76N10)},
  MRNUMBER = {3609245},
MRREVIEWER = {Olga\ S.\ Rozanova},
       DOI = {10.1007/s00205-016-1067-y},
       URL = {https://doi.org/10.1007/s00205-016-1067-y},
}

@article {MR2679372,
    AUTHOR = {Charve, Fr\'ed\'eric and Danchin, Rapha\"el},
     TITLE = {A global existence result for the compressible
              {N}avier-{S}tokes equations in the critical {$L^p$} framework},
   JOURNAL = {Arch. Ration. Mech. Anal.},
  FJOURNAL = {Archive for Rational Mechanics and Analysis},
    VOLUME = {198},
      YEAR = {2010},
    NUMBER = {1},
     PAGES = {233--271},
      ISSN = {0003-9527,1432-0673},
   MRCLASS = {35Q30 (76N10)},
  MRNUMBER = {2679372},
MRREVIEWER = {Wengu\ Chen},
       DOI = {10.1007/s00205-010-0306-x},
       URL = {https://doi.org/10.1007/s00205-010-0306-x},
}

@article {MR2847531,
    AUTHOR = {Haspot, Boris},
     TITLE = {Existence of global strong solutions in critical spaces for
              barotropic viscous fluids},
   JOURNAL = {Arch. Ration. Mech. Anal.},
  FJOURNAL = {Archive for Rational Mechanics and Analysis},
    VOLUME = {202},
      YEAR = {2011},
    NUMBER = {2},
     PAGES = {427--460},
      ISSN = {0003-9527,1432-0673},
   MRCLASS = {76N10 (35D35 35L45 35L60 35Q35)},
  MRNUMBER = {2847531},
MRREVIEWER = {Magali\ L\'ecureux-Mercier},
       DOI = {10.1007/s00205-011-0430-2},
       URL = {https://doi.org/10.1007/s00205-011-0430-2},
}

@article{li2022,
      title={Non-uniqueness for the hypo-viscous compressible {N}avier-{S}tokes equations}, 
      author={Yachun Li and Peng Qu and Zirong Zeng and Deng Zhang},
      year={2022},
      JOURNAL={arXiv:2212.05844},
      url={https://arxiv.org/abs/2212.05844}, 
}

@article {MR4643428,
    AUTHOR = {Wang, Shu and Zhang, Shuzhen},
     TITLE = {The initial value problem for the equations of motion of
              fractional compressible viscous fluids},
   JOURNAL = {J. Differential Equations},
  FJOURNAL = {Journal of Differential Equations},
    VOLUME = {377},
      YEAR = {2023},
     PAGES = {369--417},
      ISSN = {0022-0396,1090-2732},
   MRCLASS = {35Q70 (35L65 76N10)},
  MRNUMBER = {4643428},
       DOI = {10.1016/j.jde.2023.09.012},
       URL = {https://doi.org/10.1016/j.jde.2023.09.012},
}

@article {MR4897629,
    AUTHOR = {Wang, Shu and Zhang, Shuzhen},
     TITLE = {The initial value problem of the fractional compressible
              {N}avier-{S}tokes-{P}oisson system},
   JOURNAL = {J. Differential Equations},
  FJOURNAL = {Journal of Differential Equations},
    VOLUME = {438},
      YEAR = {2025},
     PAGES = {Paper No. 113359, 80},
      ISSN = {0022-0396,1090-2732},
   MRCLASS = {35Q35 (35A01 35B40 35R11)},
  MRNUMBER = {4897629},
MRREVIEWER = {Yonghui\ Zhou},
       DOI = {10.1016/j.jde.2025.113359},
       URL = {https://doi.org/10.1016/j.jde.2025.113359},
}

%\begin{thebibliography}{10}

%	\bibitem{1998Mathematical}
 % Mathematical Topics in Fluid Mechanics-Volume 2: Compressible Models
%  Lions, Pierre Louis,
%  Oxford Lecture 
%Series in Mathematics and Its Applications, vol. 10, Clarendon University Press, New York, 1998.
%  DOI:10.1063/1.3057567.
  
% \bibitem{Matsumura1980} 
% Matsumura, Akitaka  and  Nishida, Takaaki,
% The initial value problem for the equations of motion of viscous and heat-conductive gases,
%%  volume={20},
 % number={1},
%% year={1980},

\end{document}